\documentclass[a4paper,10pt]{amsart}
\usepackage{amssymb, amsmath,amsthm, mathtools,mathabx}
\usepackage{mathrsfs}
\usepackage{amscd}
\usepackage{verbatim}
\usepackage{stmaryrd}
\usepackage{bbm}
\usepackage{setspace}
\usepackage{todonotes}
\usepackage{thmtools}

\usepackage{enumerate}

\usepackage[colorlinks,linkcolor={blue},citecolor={blue},urlcolor={purple},]{hyperref}
\usepackage[nameinlink,capitalise]{cleveref}

\usepackage{tabularx}
\newcolumntype{L}{>{\arraybackslash}X}
\usepackage{multirow}

\theoremstyle{plain}
\newtheorem{theorem}{Theorem}[section]
\theoremstyle{remark}
\newtheorem{remark}[theorem]{Remark}

\theoremstyle{plain}
\newtheorem{corollary}[theorem]{Corollary}
\newtheorem{lemma}[theorem]{Lemma}
\newtheorem{proposition}[theorem]{Proposition}
\newtheorem{definition}[theorem]{Definition}

\newtheorem{assumption}[theorem]{Assumption}

\numberwithin{equation}{section}
\let\div\relax

\DeclareMathOperator{\div}{div}
\def\N{{\mathbb N}}

\def\R{{\mathbb R}}

\renewcommand{\P}{\mathbf{P}}
\newcommand{\E}{\mathbf{E}}
\newcommand{\F}{\mathscr{F}}

\newcommand{\p}{\mathbb{P}}

\newcommand{\T}{\mathbb{T}}

\newcommand{\calL}{\mathscr{L}}

\newcommand{\Ls}{\mathbb{L}}
\newcommand{\Hs}{\mathbb{H}}

\newcommand{\D}{\mathcal{D}}

\newcommand{\eps}{\varepsilon}

\renewcommand{\O}{\Omega}

\renewcommand{\o}{\mathcal{O}}

\newcommand{\wt}{\widetilde}

\newcommand{\wh}{\widehat}

\newcommand{\dd}{\mathrm{d}}
\newcommand{\Do}{\mathsf{D}}

\newcommand{\Dom}{\mathcal{O}}
\newcommand{\uin}{u_{\mathrm{in}}}

\newcommand{\ov}{\overline{v}}
\newcommand{\ophi}{\overline{\phi}}
\newcommand{\hurst}{\mathcal{H}}
\newcommand{\qstar}{q_{\hurst}}
\renewcommand{\tilde}{\wt}
\renewcommand{\hat}{\wh}

\newcommand{\embed}{\hookrightarrow}

\allowdisplaybreaks

\newenvironment{acknowledgements}{%
  \begin{abstract}
}{%
  \end{abstract}
}

\title[2D Navier-Stokes with Dirichlet boundary noise]{Global well-posedness of 2D Navier-Stokes with Dirichlet boundary fractional noise}
\author[A. Agresti]{Antonio Agresti}
\address{Delft University of Technology\\
Delft Institute of Applied Mathematics
 \\ P.O. Box 5031\\ 2600 GA Delft\\The
Netherlands
\newline
Current Address: Department of Mathematics Guido Castelnuovo, Sapienza University of Rome, P.le Aldo Moro 5, 00185 Rome, Italy} 
\email{antonio.agresti92@gmail.com}

\author[A. Blessing]{Alexandra Blessing (Neam\c tu)}
\address{University of Konstanz, Department of Mathematics and Statistics, Universit\"atsstraße 10, 78464 Konstanz, Germany} 
\email{alexandra.blessing@uni-konstanz.de}

\author[E. Luongo]{Eliseo Luongo}
\address{Scuola Normale Superiore, Piazza dei Cavalieri, 7, 56126 Pisa, Italia \newline
Current Address: Fakultät für Mathematik, Universität Bielefeld, 33501 Bielefeld, Germany} 
\email{eliseo.luongo@sns.it}

\thanks{A.A. has received funding from the VICI subsidy VI.C.212.027 of the Netherlands Organisation for Scientific Research (NWO). A.A. and E.L. are members of GNAMPA (IN$\delta$AM)}
\thanks{A.B. acknowledges support from DFG CRC/TRR 388, project A06.}

\date\today

\begin{document}

\begin{abstract}
In this paper, we prove the global well-posedness and interior regularity for the 2D Navier-Stokes equations driven by a fractional noise acting as an inhomogeneous Dirichlet-type boundary condition.
The model describes a vertical slice of the ocean with a relative motion between the two surfaces and can be thought of as a stochastic variant of the Couette flow. The relative motion of the surfaces is modeled by a Gaussian noise which is coloured in space and fractional in time with Hurst parameter $\hurst>\frac{3}{4}$.
\end{abstract}

\maketitle

\tableofcontents

\section{Introduction}\label{s:intro}
In many situations occurring in applied sciences, noise can affect the evolution of a system only through the boundary of a region where the system evolves. Such phenomena can be modeled via partial differential equations with boundary noise, as introduced by Da Prato and Zabczyck in the seminal paper \cite{da1993evolution}. Such a description presents several issues from a mathematical viewpoint. Indeed, nowadays it is well-known that in the one dimensional case, the solution of the heat equation with white noise Dirichlet or Neumann boundary conditions has low (space) regularity compared to the case of noise diffused inside the domain. This is due to the large amplitude of the fluctuations of the solutions close to the boundary. In particular, in the case of Dirichlet boundary conditions the solution is only a distribution. This allowed to treat only a restricted class of nonlinearities, exploiting specific properties of the heat semigroup and studying carefully the blow-up of the solution close to the boundary. For some results in this direction, the reader is referred to \cite{alos2002stochastic,brzezniak2015second,fabbri2009lq,Goldys23}.
On the contrary, partial differential equations with white noise Neumann boundary conditions with more severe nonlinearities have been considered in \cite{Cerrai,sowers1994multidimensional}, and in the last few years, maximal $L^p$ regularity techniques provided new ideas to treat some of those arising in fluid dynamics. Indeed, some results on the global and local well-posedness of the 2D Navier-Stokes equations and the 3D primitive equations with boundary noise perturbations of Neumann type have been proven in \cite{AL23_boundary_noise} and \cite{Binz_BN}, respectively. Besides the physical interests in studying the Navier-Stokes equations with boundary noise in virtue of its connection with the Couette flow (see also below for further motivations), the present manuscript also aims at (partially) filling the gap in the literature between Dirichlet and Neumann type boundary conditions for fluid dynamical models.

\smallskip

Throughout the manuscript, we fix a finite time horizon $T>0$ and consider a spatial domain $\Dom=\mathbb{T}\times(0,a)$ where $\mathbb{T}$ is the one-dimensional torus and $a>0$. We further define the lower and upper parts of the boundary of $\Dom$ by
\begin{equation}
\Gamma_b=\mathbb{T}\times\{0\}\quad\text{ and }\quad\Gamma_u=\mathbb{T}\times\{a\}.
\end{equation} 
In this work, we focus on the global well-posedness and interior regularity of the two-dimensional Navier–Stokes equations with fractional boundary noise. The unknowns are the velocity field $u(t,\omega,x,z)=(u_1,u_2):(0,T)\times \Omega\times \Dom \to \R^2$ and the pressure $P:(0,T)\times \O\times \Dom\to \R$, which formally satisfy the system
\begin{equation}
\label{eq:NS_fractional_noise}
\left\{
\begin{aligned}
\partial_t u 
&=\Delta u+\nabla P -(u\cdot\nabla) u   \quad &\text{ on }&(0,T)\times\Dom,\\
\operatorname{div}u &=0 \quad &\text{ on }&(0,T)\times\Dom,\\
u_1 &=g\,\dot{W}^\hurst\quad &\text{ on }&(0,T)\times\Gamma_u,\\
u_2 &=0\quad &\text{ on }&(0,T)\times\Gamma_u,\\
u&=0\quad &\text{ on }&(0,T)\times\Gamma_b,\\
u(0)&=\uin\quad &\text{ on }&\Dom,
\end{aligned}
\right.
\end{equation}
where $(\uin,g)$ are given data and $W^\hurst$ is a fractional Brownian motion with Hurst parameter $\hurst>\frac{3}{4}$, respectively. The assumptions on $(\uin,g,W^\hurst)$ are made precise below. Even if we consider a more regular noise in time than the one introduced in \cite{da1993evolution}, the combination of the blow-up of the solution close to the boundary and the Navier-Stokes nonlinearity makes the global well-posedness and the interior regularity of \eqref{eq:NS_fractional_noise} a non-trivial issue, which, indeed, cannot be treated simply by the techniques introduced in \cite{AL23_boundary_noise}. Indeed, to the best of our knowledge, this is the first instance of a \emph{global} well-posedness result for a fluid dynamical system with non-homogeneous Dirichlet-type boundary conditions of a regularity class comparable with the time derivative of a fractional Brownian motion with Hurst parameter $\hurst> \frac{3}{4}$, see \cite{chang2015solvability,farwig2011global,grubb2001nonhomogeneous} and the references therein for some results in this direction. Moreover, the reader is referred to \cite{Wai-Tong21} for the analysis of some properties of \eqref{eq:NS_fractional_noise} in the 3D case replacing $g\,\dot{W}^\hurst$ with an Ornstein Uhlenbeck process and to \cite{alos2002stochastic,Goldys23} for some results on the existence and uniqueness of solutions for the heat equation with white noise Dirichlet type boundary conditions perturbed by some Lipschitz forcing. Finally, in \cite{berselli2006existence,raymond2007stokes} the emphasis is on the non-penetration boundary conditions, namely it is studied the case $u_2=g(x,t) $ on $\Gamma_u\cup \Gamma_b$, with $g$ much more regular either in time and space than $g\,\dot{W}^\hurst$. 

According to \cite{gill1982atmosphere,pedlosky1996ocean, pedlosky2013geophysical}, see also the discussion in the introduction of \cite{AL23_boundary_noise}, the geometry considered in \eqref{eq:NS_fractional_noise} can be seen as an idealization of the ocean dynamics (more precisely, a vertical slice of the ocean). The model \eqref{eq:NS_fractional_noise}, describes a Couette flow, namely a viscous fluid in the space between two surfaces, one of which is moving tangentially relative to the other (see also \Cref{rem:extension}). The relative motion of the surfaces imposes a shear stress on the fluid and induces the flow. Let us recall that the onset of turbulence is often related to the randomness of background movement \cite{mikulevicius2004stochastic}. Moreover, according to \cite[Chapter 3]{pope2001turbulent} in any turbulent flow there are unavoidably perturbations in boundary conditions and material properties. We model these features by the noise term $g\,\dot{W}^\hurst$. As introduced by Kolmogorov in~\cite{Kolmogorov}, fractional Brownian motion can be thought of as a model for turbulence. 
Moreover, to describe turbulence in 3D fluids, models of random vortex filaments have been introduced in~\cite{FlandoliGubinelli}. These have been analyzed for fractional Brownian motion with $\hurst>1/2$ in~\cite{NualartVortex} and $\hurst<1/2$ in~\cite{FlandoliFBM}.

\subsection{Main result}\label{sec:main results}
We begin by introducing some basic notation. 
Throughout this manuscript, we work on a complete filtered probability space $(\Omega,\mathcal{F},(\mathcal{F}_t)_{t\geq0},\mathbf{P})$ and consider a separable Hilbert space $U$. A process 
$\Phi$ is said to be $\mathcal{F}$-progressive measurable if, for every $t>0$, the restriction  $\Phi|_{(0,t)\times \Omega}$ is measurable with respect to $\mathcal{F}_t\otimes \mathcal{B}((0,t))$, where $\mathcal{B}$ denotes the Borel $\sigma$-algebra. Further notation concerning function spaces is deferred to \Cref{sss:notation}. On the noise $W^\hurst$ we enforce the following 
\begin{assumption}
\label{ass:ass_fractional}
$W^\hurst$ is a $U$-cylindrical fractional Brownian motion with Hurst parameter $\hurst\in (\frac{3}{4},1)$
and $
g\in \calL_2 (U,H^{-s}(\Gamma_u))$ with $ s\in [0,\frac{1}{2})$ and  $\hurst-\frac{s}{2}>\frac{3}{4}$.
\end{assumption}
Note that \Cref{ass:ass_fractional} is consistent with the results obtained in~\cite{Maslowski} for the stochastic heat equation with Dirichlet fractional noise. The reader is referred to \Cref{r:generalization} for the case of a time-dependent $g$. 
\smallskip
Following \cite[Chapter 15]{DPZ_ergodicity} and \cite{da2002two}, {we construct solutions to \eqref{eq:NS_fractional_noise} by the splitting}
\begin{align}\label{split_intro}
    u= w_g + v,
\end{align}
where $w_g$ is a mild solution of the linear problem with non-homogeneous boundary conditions
\begin{equation}
\label{Linear stochastic}
\left\{
\begin{aligned}
\partial_t w_g 
&=\Delta w_g+\nabla P_g    \ \ &\text{ on }&(0,T)\times\Dom,\\
\operatorname{div} w_g &=0 \ \ &\text{ on }&(0,T)\times\Dom,\\
w_{g,1} &=g\,\dot{W}^\hurst \ \ &\text{ on }&(0,T)\times\Gamma_u,\\
{w_{g,2}} &=0\ \ &\text{ on }&(0,T)\times\Gamma_u,\\
w_g&=0\ \ &\text{ on }&(0,T)\times\Gamma_b,\\
w_g(0)&=0\ \ &\text{ on }&\Dom
\end{aligned}
\right.
\end{equation}
and $v$ is a weak solution of
\begin{equation}
\label{eq:modified_intro}
\left\{
\begin{aligned}
\partial_t v& = \Delta v +\nabla(P-P_g)\\
&\quad\quad \ \  -\div ((v+w_g)\otimes (v+w_g))   \ \ &\text{ on }&(0,T)\times\Dom,\\
\operatorname{div}v&=0 \ \ &\text{ on }&(0,T)\times\Dom,\\
v&=0\ \ &\text{ on }&(0,T)\times\left(\Gamma_b\cup\Gamma_u\right),\\
v(0)&=\uin\ \ &\text{ on }&\Dom.
\end{aligned}
\right.
\end{equation}
In \eqref{eq:modified_intro}, due to the divergence-free of $v$ and $w_g$, we rewrote the Navier-Stokes nonlinearity in the conservative form to accommodate the weak (PDE) setting.

As discussed in \cite[Chapter 13]{DPZ_ergodicity}, if $g$, $\uin$, $W^\hurst(t)$ were sufficiently regular, then $u=v+w_g$ would be a classical solution of the Navier-Stokes equations with non-homogeneous boundary conditions \eqref{eq:NS_fractional_noise}.

Next, we introduce the class of solutions we are going to consider. To motivate them, let us first discuss the regularity of $w_g$. It is well-known that, in the case of Dirichlet-type boundary conditions, the solution of a linear problem with boundary noise and $\hurst=\frac{1}{2}$ is a distribution which blows-up close to the boundary, see \cite{alos2002stochastic,da1993evolution}, and the same holds also in case of $\hurst\neq \frac{1}{2}$, see \cite{brzezniak2015second}. Therefore, we cannot expect that the mild solution of \eqref{Linear stochastic} has arbitrarily good integrability properties as in \cite{AL23_boundary_noise, Binz_BN}. This has drastic consequences in our analysis. As we will show in \Cref{regularity stokes}, we have, $\P-a.s.$,
\begin{equation}
\label{eq:pathwise_regularity_wg_intro}
w_g \in C([0,T];L^{2q}(\Dom;\R^2))\ \text{ for all }q\in (1,\qstar)
\end{equation}
where 
\begin{equation}
\label{eq:def_qstar}
 \qstar:= \frac{2}{2s+5-4\hurst}\in (1,2).
\end{equation}
Let us stress that $\qstar<2$ and $\lim_{\hurst \downarrow 3/4}q_\hurst{=}1$ even if $s=0$. As we will see below, this fact creates major difficulties in our analysis of the auxiliary Navier-Stokes equations \eqref{eq:modified_intro}.
In particular, $w_g\otimes w_g\in C([0,\infty);L^{q}(\Dom;\R^2))$ $\P-a.s.$ and, from parabolic regularity, the best regularity we can hope for is $v\in L^p([0,\infty);H^{1,q}(\Dom;\R^2))$ 
$\P-a.s.$ for all $p<\infty$. Thus, in general,
$$
v\not\in L^2(0,T;H^{1}(\Dom;\R^2))\ \ \P-a.s. \text{ for any }T<\infty.
$$
Therefore, $v$ is a solution of the Navier-Stokes equations with \emph{infinite energy} and the argument used in \cite{AL23_boundary_noise} does not work.
The case of infinite energy solutions of 2D Navier-Stokes equations already appeared in the literature \cite{BF09_rough_force,GP02_infinite_energy}. In \cite{GP02_infinite_energy} the unboundedness of the energy is due to a rough initial data $u_0\not\in L^2$ while in \cite{BF09_rough_force} to a rough forcing term $f\not\in L^2(0,T;H^{-1})$ acting on the bulk. Our case does not 
fit in any of the above situations due to the presence of transport-type terms depending on the $w_g$ in \eqref{eq:modified_intro} and the fact that we are working on domains. For this reason, our proofs rely on different methods. For details, the reader is referred to the text before \Cref{r:generalization}.

\smallskip

In light of the previous discussion, we are now ready to define solutions to \eqref{eq:NS_fractional_noise}. 
Below, we set $A:B=\sum_{i,j=1}^2 A^{i,j}B^{i,j}$ for two matrices $A$ and $B$ and $\mathbb{L}^q$ the image of $L^q(\Dom;\R^2)$ via the Leray projection $\mathbb{P}$ defined rigorously in \Cref{subsec:stokes operator }.
\begin{definition}\label{def z weak sol}
Let $T<\infty$, $\uin\in L^0_{\F_0}(\O;\Ls^2)$ and $q\in (1,\qstar)$.
\begin{itemize}
    \item {\rm ($q$-solution)} A progressively measurable process $u$ with $\mathbf{P}-a.s.$ paths in 
$L^{2q'}(0,T;\Ls^{2q})$, is a pathwise weak $q$-solution of \eqref{eq:NS_fractional_noise} if for all divergence-free $\varphi=(\varphi_1,\varphi_2)\in C^{\infty}(\Dom;\R^2)$ such that 
$\varphi=0$ on $\Gamma_{b}\cup \Gamma_u
$ and a.e.\ $t\in (0,T)$,
\begin{align*}
&\int_{\Dom} u(x,t)\varphi(x)\,\dd x 
-
\int_{\Dom} \uin(x)\varphi(x)\,\dd x \\
&\qquad\qquad 
=
\int_0^t \int_{\Dom}\big(u\cdot \Delta \varphi + [u\otimes u]:\nabla \varphi\big)\,\dd x\, \dd r-{\langle g,\partial_2\varphi_1\rangle_{H^{-s}(\Gamma_u),H^{s}(\Gamma_u)}W^\hurst_t }.
\end{align*}
\item {\rm (unique $q$-solution)} 
A $q$-solution $u$ to \eqref{eq:NS_fractional_noise} is said to be a \emph{unique solution} if for any other $q$-solution $\widetilde{u}$ we have $u=\widetilde{u}$ a.e.\ on $[0,T]\times \Omega$. 
\item {\rm (unique solution)} A $q$-solution $u$ is said to be a {\rm unique solution} to \eqref{eq:NS_fractional_noise} if it is also a $\widetilde{q}$-solution for all $\widetilde{q}\in (1,\qstar)$.
\end{itemize}
\end{definition}

Before stating our main result, let us first comment on the above definition.
Due to the argument below \eqref{eq:modified_intro}, one cannot expect solutions to \eqref{eq:NS_fractional_noise} with integrability in space larger or equal to $2\qstar$. Furthermore, {even if solutions are constructed by \eqref{split_intro}, the uniqueness class is independent of this splitting. Moreover, the} unique solution of \eqref{eq:NS_fractional_noise} is independent of the choice of $q\in (1,\qstar)$. Such independence is expected from solutions to \eqref{eq:NS_fractional_noise} in light of \eqref{eq:pathwise_regularity_wg_intro}.
Finally, let us discuss the regularity class chosen to define $q$-solutions. Since $\Dom$ is two-dimensional, the space $L^{2q'}(0,T;L^{2q}(\Dom;\R^2))$ has Sobolev index given by (keeping in mind the parabolic scaling) 
$$
-\frac{2}{2q'}-\frac{2}{2q}=-1.
$$
In particular, the regularity class chosen for $q$-solutions to \eqref{eq:NS_fractional_noise} is \emph{critical} for the Navier-Stokes equations in two dimensions and satisfies the classical Ladyzhenskaya–Prodi–Serrin condition.
In light of the recent convex integration results \cite{CL22_sharpnonuniqueness,CL23_nonuniqueness,LT24_nonuniqueness} in absence of noise and with periodic boundary conditions in all directions, the regularity assumption in our definition is expected to be sharp for obtaining uniqueness and a-fortiori well-posedness.

\smallskip

The main result of the current work reads as follows. 

\begin{theorem}
\label{t:global}
Let \Cref{ass:ass_fractional} be satisfied and
$    \uin\in L^0_{\F_0}( \O;\Ls^2)$.
\begin{enumerate}[{\rm(1)}]
    \item\label{it:global1}
    There exists a \emph{unique solution} of \eqref{eq:NS_fractional_noise} in the sense of \Cref{def z weak sol} with paths in
$$
u\in C([0,T];\Ls^2) \ \  \mathbf{P}-a.s.
$$
\item\label{it:global2}
The unique solution of \eqref{eq:NS_fractional_noise} satisfies, for all $t_0\in (0,T)$ and 
$\Dom_0\subset \Dom $ such that $\mathrm{dist}({\Dom_0}, \partial \Dom)>0$,
    \begin{align*}
        u\in C([t_0,T]; C^{\infty}({\Dom_0};\mathbb{R}^2)) \ \  \mathbf{P}-a.s.
    \end{align*}
\end{enumerate}
\end{theorem}

The proof of \Cref{t:global}\eqref{it:global1} and \eqref{it:global2} are given in \Cref{s:uniqueness_q_solution} and \Cref{sec interior regularity nonlinear auxiliary}, respectively. Routine extensions of the above are commented in \Cref{r:generalization} below.

\smallskip

Next, let us discuss the main ideas behind the proof of \Cref{t:global}. As commented above, due to \eqref{eq:pathwise_regularity_wg_intro}, we cannot deal with the techniques introduced in \cite{AL23_boundary_noise} to study \eqref{eq:modified_intro}. Indeed, contrary to  \cite{AL23_boundary_noise,da2002two}, the splitting introduced above is \emph{not} enough to study the global well-posedness of \eqref{eq:NS_fractional_noise} since \eqref{eq:modified_intro} has no Leray solutions since $w_g\otimes w_g\not \in L^2(0,T;L^2)$. Thus, we control the blow-up of the energy of $v$ introducing further splittings depending on the regularity of $w_g$. As discussed above we will show that $w_g\in C([0,T];L^{2q})$ for some $q\in (1,q_{\mathcal{H}})$. Since the space $C([0,T];L^{2q})$, $q>1$ is subcritical for 2D Navier-Stokes equations we have some hopes to exploit the strong time regularity of $w_g$ to circumvent its rough behaviour in space. The heuristic idea above is realized by writing 
\begin{align}
\label{eq:splitting_v_intro}
    v=\sum_{i=0}^{N-1}v_i+\ov,
\end{align} 
where $N$ depends only on $q$. The terms $\{v_i\}_{i\in \{0,\dots, N-1\}}$ are defined inductively solving homogeneous Stokes equations with forcing having mixed regularity in space and time, such that the regularity in space increases in $i$ while the regularity in time decreases in $i$. On the contrary $\ov$ is a Leray-type solution of the remainder equation. In particular, $N$ is chosen large enough such that the equation for $\ov$ has a forcing in $L^2(0,T;H^{-1})$ and therefore is regular enough to prove the existence and uniqueness of Leray solutions, see \Cref{Thm deterministic well--posed} below. 
The reader is referred to \Cref{subsec: nonlinear aux}
for further discussions.

The interior regularity of $u$ in \Cref{t:global}\eqref{it:global2} is treated considering again the splitting $u=v+w_g$ introduced above. The interior regularity of $w_g$ can be proved similarly to the linear part of \cite{AL23_boundary_noise}. On the contrary, the low regularity of $v$ does not allow us to study directly its interior regularity by Serrin's argument as in \cite{AL23_boundary_noise}. For this reason, we rely on the splitting \eqref{eq:splitting_v_intro} analysed in \Cref{subsec: nonlinear aux} to study the well-posedness of \eqref{eq:modified_intro}. 
Combining maximal $L^p$ regularity techniques for studying the interior regularity of the $v_i$'s, an induction argument and a Serrin argument for treating the interior regularity of $\ov$, we obtain the required regularity of $v$. As shown in \cite{serrin1961interior} (see also \cite[Section 13.1]{lemarie2018navier}) and similarly to \cite{AL23_boundary_noise}, higher-order interior time regularity does not appear to be attainable for our Navier–Stokes problem with stochastic boundary conditions. This stands in contrast to the case of the heat equation with white noise boundary conditions, as studied in \cite{brzezniak2015second}. The underlying reason for this phenomenon is the presence of the unknown pressure $P$, whose non-local nature creates a link between interior and boundary regularity, where the noise acts.
\smallskip

To conclude, let us point out that, in contrast to \cite{BF09_rough_force,GP02_infinite_energy}, we employ a different splitting scheme to prove existence due to the presence of the transport-type terms originated by $w_g$. Moreover, the number of splitting $N$ depends on how much the Sobolev index of the space $C([0,T];L^{2q}(\Dom;\R^2))$, i.e.\ $-\frac{1}{q}$, is far from the critical threshold $-1$. In particular, $N\to \infty$ as $q\downarrow 1$. As commented above, such a splitting is also convenient when proving the interior regularity for $u$ which was not addressed in the above-mentioned works.

\begin{remark}[Extensions]
\label{rem:extension}
One can readily check that 
\Cref{t:global} extends in the following cases:
\label{r:generalization}
\begin{itemize}
    \item (Bounded domains) If $\Dom$ is replaced by  a smooth $C^2$-bounded domain in $\R^2$. However, we prefer to keep the same geometry of \cite{AL23_boundary_noise} for two reasons. Firstly, and more importantly, as discussed in \Cref{s:intro}, the model considered has a clear physical interpretation. Secondly, in this way, we can easily compare our results, techniques and assumptions with those of \cite{AL23_boundary_noise}.
    \item (Fractional Volterra noise) 
If $W^\hurst$ is replaced by a $\alpha$-regular Volterra process with $\alpha>\frac{1}{4}$.
Let us recall that a fractional Brownian motion with Hurst parameter $\hurst$ is an example of a $\alpha$-regular Volterra process with $\alpha=\hurst -\frac{1}{2}$. These are non-Markovian stochastic processes which can be represented as integrals of kernels with respect to the Brownian motion and include for example the fractional Liouville Brownian motion and the Rosenblatt process.~Stochastic convolutions with respect to such processes were analyzed in~\cite{BvNS12,MaslowskiC,MCO}. 
\item (Time-dependent $g$)
The term
$g$ in the boundary noise $g \dot{W}^\hurst$ depends on time as long as it is progressively measurable and the corresponding process $w_g$ satisfies \eqref{eq:pathwise_regularity_wg_intro}. 
\item {(Full stochastic Couette flow) 
More general boundary conditions like \begin{align*}
    \begin{cases}
        u_1=U_{up}+g\dot{W}^{\mathcal{H}}\quad &\text{on } (0,T)\times \Gamma_u,\\
    u_1=U_{b}\quad &\text{on } (0,T)\times \Gamma_b.
    \end{cases}
\end{align*}
can be considered for sufficiently smooth velocity fields $U_{up}, U_b$ such that the corresponding process $w_g$ satisfies \eqref{eq:pathwise_regularity_wg_intro}. For example the case $U_{up}, U_b\in L^2((0,T)\times \mathbb{T})$ can be treated. The above can be seen as a Couette flow with uncertainty on the velocity of one of the two surfaces. 
}
\end{itemize}
\end{remark}


\subsection{Overview}
In \Cref{sec preliminaries}, we introduce the functional framework required to study problem \eqref{eq:NS_fractional_noise}.
The proof of \Cref{t:global} is developed in \Cref{sec well posed} and \Cref{sec:interior reg}. Specifically, global well-posedness, i.e. \Cref{it:global1}, is addressed in \Cref{sec well posed}, where we first analyze the linear problem \eqref{Linear stochastic} in \Cref{sec regularity mild stokes}, followed by the nonlinear problem \eqref{eq:modified_intro} in \Cref{subsec: nonlinear aux}.
Interior regularity, i.e. \Cref{it:global2}, is the focus of \Cref{sec:interior reg}. In particular, \Cref{sec interior regularity mild stokes} is devoted to the interior regularity of the solution to the linear problem \eqref{Linear stochastic}, while \Cref{sec interior regularity nonlinear auxiliary} deals with the nonlinear problem \eqref{eq:modified_intro}.

\subsection{Notation} 
\label{sss:notation}
Here we collect some notation which will be used throughout the paper. Additional notation will be introduced where needed. We use $C$ to denote a generic constant, which may vary from line to line. When it is important to emphasize the dependence of $C$ on a parameter $\xi$, we write $C(\xi)$. Moreover, we sometimes write $a \lesssim b$ (resp.\ $a\lesssim_\xi b$) to mean that there exists a constant $C$ (resp. $C(\xi)$) such that $a \leq C b$ (resp.\ $a\leq C(\xi) b$).

Let $q\in (1,\infty)$ be fixed. For any integer $k\geq 1$, we denote by $W^{k,q}$ the standard Sobolev spaces. In the case of non-integer smoothness $s\in (0,\infty)\setminus \mathbb{N}$, we define $W^{s,q}=B^s_{q,q}$ where $B^{s}_{q,q}$ is the Besov space with smoothness $s$, and integrability $q$ and microscopic integrability $q$. We also denote by $H^{s,q}$ the Bessel potential spaces. Both Besov and Bessel potential spaces can be defined via Littlewood–Paley theory (see, e.g., \cite{ST87}, \cite[Section 6]{Saw_Besov}), or through interpolation methods based on the classical Sobolev spaces $W^{k,q}$ (see, e.g., \cite[Chapter 6]{BeLo}).
For a domain $D\subset \R^n$, integer $d\geq 1$, and $\mathcal{A}\in \{W,H\}$, we define the vector-valued spaces by $\mathcal{A}^{s,q}(D;\R^d)=(\mathcal{A}^{s,q}(D))^d$.

Let $\mathcal{K}_1$ and $\mathcal{K}_2$ be two separable Hilbert spaces. We denote by $\calL_2(\mathcal{K}_1,\mathcal{K}_2)$ the set of Hilbert-Schmidt operators from $\mathcal{K}_1$ to $\mathcal{K}_2$. 
We will use the following Fubini-type identity:
$$
H^s(D;\mathcal{K}_1)=\calL_2(\mathcal{K}_1,H^s(D))\ \ \text{ for all }s\in\R,
$$
which follows from \cite[Theorem 9.3.6]{Analysis2} and interpolation theory.

\section{Preliminaries}\label{sec preliminaries}

\subsection{The Stokes operator and its spectral properties}\label{subsec:stokes operator }
In this section, we introduce the functional analytic setup to define all the objects necessary in the following. 
Throughout this subsection, we let $r\in (1,\infty)$. Recall that $\o=\T\times (0,a)$ where $a>0$.\\
We begin by introducing the Helmholtz projection on $L^r(\Dom;\R^2)$, see e.g.\ \cite[Subsection 7.4]{pruss2016moving}. As described also in \cite[Section 2.1]{AL23_boundary_noise}, the projection can be defined via an elliptic problem. We recall its construction here for the sake of completeness. Let $f\in L^r(\Dom;\R^2)$ and let $\psi_f$ be the unique solution of the following elliptic problem
\begin{equation}
\label{eq:psi_f_problem}
\left\{
\begin{aligned}
\Delta  \psi_f &= \div f\quad &\text{ on }&\Dom,\\
\partial_{\hat{n}}  \psi_f &= f\cdot \hat{n}   &\text{ on }&\Gamma_u \cup \Gamma_b.
\end{aligned}
\right.
\end{equation}
Here $\hat{n}$ denotes the exterior normal vector field on $\partial\Dom$. This problem is understood in its standard weak formulation:
\begin{equation}
\int_{\Dom} \nabla \psi_f \cdot \nabla \varphi \,\dd x\dd z = \int_{\Dom} f\cdot \nabla \varphi\,\dd x \dd z\ \  \text{ for all }    \varphi\in C^{\infty}(\Dom).
\end{equation}
By \cite[Corollary 7.4.4]{pruss2016moving} , we have $\psi_f\in W^{1,r}(\Dom)$ and $\|\nabla \psi_f\|_{L^{r}(\Dom;\R^2)}\lesssim \|f\|_{L^r(\Dom;\R^2)}$. The Helmholtz projection $\p_r:L^r(\Dom;\R^2)\rightarrow L^r(\Dom;\R^2)$ is then defined by 
\begin{equation*}
\p_r f= f- \nabla \psi_f, \quad f\in L^r(\Dom;\R^2).    
\end{equation*}
Next, we define the Stokes operator on $L^r(\Dom;\R^2)$ corresponding to the boundary conditions considered in \eqref{eq:NS_fractional_noise}. For notational convenience, we define $A_r$ as minus the Stokes operator so that $A_r$ is a positive operator for $r=2$ (i.e.\ $\langle A_2 u,u \rangle\geq 0$ for all $u\in \Do(A_2)$). Let 
\begin{equation*}
\Ls^r:=\p(L^r(\Dom;\R^2)),\quad \Hs^{s,r}:=H^{s,r}(\Dom;\R^2)\cap \Ls^r, \ \ s\in \R.    
\end{equation*} Then, we define the operator $A_r:\Do(A_r)\subseteq \Ls^r\to \Ls^r$ where
\begin{align*}
\Do(A_r)= \big\{f=(f_1,f_2)\in W^{2,r}(\Dom;\mathbb{R}^2)\cap \Ls^r\,:\,
\ &   f|_{\Gamma_b\cup \Gamma_u}=0\big\},
\end{align*}
and  $A_r u=-\p \Delta u$ for $u\in \Do(A_r)$.

In the main arguments, we need stochastic and deterministic maximal $L^r$-regularity estimates for convolutions. By \cite{MaximalLpregularity, KuWe}, it is enough to provide the boundedness of the $H^{\infty}$-calculus for $A_r$. The reader is referred to \cite[Chapters 3 and 4]{pruss2016moving} and \cite[Chapter 10]{Analysis2} for the main notation and basic results on the $H^{\infty}$-calculus.\\ 
Contrary to \cite{AL23_boundary_noise}, the boundary conditions we are interested in here are much more classical. Indeed, the Stokes operator with no-slip boundary conditions is well-studied. The reader is referred e.g. to \cite[Section 2.8]{hieber2018stokes}, \cite{NS03}, \cite{Gigafract} and \cite[Section 9]{KKW} for the proof of this nowadays classical statement.
\begin{lemma}
\label{lem:fractional_powers_stokes}
For all $r\in (1,\infty)$, the operator $A_{r}$ is invertible and has a bounded $H^{\infty}$-calculus of angle $0$. Moreover, the domain of the fractional powers of $A_r$ is characterized as follows:
$$
\Do(A_{r}^{\alpha})
= 
\left\{
\begin{aligned}
&\Hs^{2\alpha,r}\ \  &\text{ if }& \ \alpha<\frac{1}{2r}, \\
&\big\{u\in\Hs^{2\alpha,r}\,:\,u|_{\partial\Dom}=0\big\} \ \  &\text{ if }& \ \frac{1}{2r}<\alpha\leq 1.
\end{aligned}
\right.
$$
\end{lemma}
The above implies that $-A_r$ generates an analytic semigroup on $\Ls^r$ which admits stochastic and deterministic maximal $L^p$-regularity for all $p\in (1,+\infty)$, see \cite[Chapter 3-4]{pruss2016moving} and \cite{MaximalLpregularity}. We denote such semigroup by $S_r(t)$. We continue introducing some known facts about the ``Sobolev tower'' of spaces associated with the operator $A_r$. We denote by 
\begin{align*}
X_{\alpha,A_r}&= \Do(A_r^{\alpha}) \  &\text{ for }&\alpha\geq 0,\\
X_{\alpha,A_r}&= (\Ls^r,\|A_r^{{\alpha}} \cdot\|_{\mathbb{L}^r})^{\sim} \  &\text{ for }&\alpha<0,
\end{align*}
where $\sim$ denotes the completion. Indeed, since $0\in \rho(A_r)$ by \Cref{lem:fractional_powers_stokes}, we have that $f\mapsto \|A_r^{{\alpha}} f\|_{\Ls^r}$ is a norm for all $\alpha<0$. 
Since $(A_r)^*=A_{r'}$, it follows that (see e.g.\ \cite[Chapter 5, Theorem 1.4.9]{Am})
\begin{equation}
\label{eq:duality_fractional_scale}
(X_{\alpha,A_r})^*= X_{-\alpha,A_{r'}}.
\end{equation}
For notational convenience, we will write $A,\ S(t)$ instead of $A_2$ and $S(t)$. Moreover we define 
\begin{align*}
    H:=\mathbb{L}^2,\quad V:=\Do(A^{1/2}).
\end{align*}
We denote by $\langle\cdot,\cdot\rangle$ and $\lVert\cdot\rVert$ the inner product and the norm in $H$, respectively.
In the following, $V^*$ denotes the dual of $V$ and we identify $H$ with its dual $H^*$. Whenever $X$ is a reflexive Banach space such that the embedding $X\hookrightarrow H$ is continuous and dense, denoting by $X^*$ the dual of $X$, the scalar product $\left\langle
\cdot,\cdot\right\rangle $ in $H$ extends to the dual pairing between $X$ and $X^{*}$. We will simplify the notation accordingly.\\

{
For the convenience of the reader, we recall here the definition of the deterministic maximal $L^p$-regularity, since our techniques heavily rely on it. 
\begin{definition}{\em (\cite[Definition 3.5.1]{pruss2016moving})}
    We let $\mathcal{X}$ stand for a Banach space and $\mathcal{A}$ a linear closed operator on $\mathcal{X}$ with domain $\Do(\mathcal{A})$. We say that the inhomogeneous initial value problem on $L^p(0,T;\mathcal{X})$  given by 
\begin{align}\label{max:reg}
\dot{u}(t) + \mathcal{A} u(t) =f(t), ~~u(0)=u_0 
\end{align}
admits maximal $L^p$-regularity, if for each $f\in L^p(0,T;\mathcal{X})$ and $u_0\in (\mathcal{X},\Do(\mathcal{A}))_{1-1/p,p}$ there exists a unique $u\in W^{1,p}(0,T;\mathcal{X})\cap L^p(0,T;\Do(\mathcal{A}))$ satisfying~\eqref{max:reg} a.e.\ in $(0,T)$.
\end{definition}
}

\subsection{The Dirichlet map}
Now we are interested in $L^2$-estimates for the Dirichlet map, i.e.\ we are interested in studying the weak solutions of the elliptic problem
\begin{equation}\label{Non Homogeneous Stokes}
\left\{
\begin{aligned}
-\Delta u+\nabla \pi   &=0, & \text{ on }& \Dom,\\
\operatorname{div} u   &=0\qquad & \text{ on }& \Dom,  \\
u(  \cdot,0)     &=0, & \text{ on }& \Gamma_b,\\
u_1(  \cdot ,a)  & =g, & \text{ on }& \Gamma_u,\\
u_2     &=0, & \text{ on }& \Gamma_u.
\end{aligned}\right.
\end{equation}
To state the main result of this subsection, we formulate \eqref{Non Homogeneous Stokes} in the very weak setting. To this end, we argue formally. Take $\varphi=(\varphi_1,\varphi_2)\in C^{\infty}(\Dom;\R^2)$ such that $\operatorname{div} \varphi=0$, 
$$
\varphi=0, \quad\text{ on } \quad \Gamma_{b}\cup \Gamma_u.
$$
A formal integration by parts shows that \eqref{Non Homogeneous Stokes} implies
\begin{equation}
\label{Non Homogeneous Stokes weak}
    \int_{\Dom}  u\cdot \Delta \varphi \,\dd x \dd z= \int_{\T} g(x) \partial_2\varphi_1(x,a)\,\dd x.
\end{equation}

In particular, the RHS of \eqref{Non Homogeneous Stokes weak} makes sense even in case $g$ is a distribution if we interpret $\int_{\T} g(x) \partial_2\varphi_1(x,a)\,\dd x =\langle \partial_2\varphi_1(\cdot ,a),g\rangle$. The well-posedness of \eqref{Non Homogeneous Stokes} is, as for the properties of the Stokes operator, a well-known fact. Indeed, \Cref{Regularity of the Dirichlet map} below holds. The reader is referred to \cite{amrouche1994decomposition,amrouche2010very,farwig2005very,frohlich2007stokes,schumacher2008very} for its proof and more general results on the Dirichlet boundary values problem above even in case of weighted $L^r$ spaces of Muckenhoupt class and $u\cdot \hat{n}|_{\Gamma_u\cup \Gamma_b}\neq 0$.
\begin{theorem}
    \label{Regularity of the Dirichlet map}
For all $g\in H^{-\frac{1}{2}}(\Gamma_u)$ there exists a unique very weak solution $(u,\pi)\in H\times H^{-1}(\Dom)/ \mathbb{R}$ of \eqref{Non Homogeneous Stokes}. Moreover $(u,\pi)$ satisfy 
\begin{align}
 \label{eq:elliptic_regularity_1}
     \lVert u \rVert+\lVert \pi \rVert_{H^{-1}(\Dom)/ \mathbb{R}}\leq C\lVert g\rVert_{H^{-\frac{1}{2}}(\Gamma_u)}.
 \end{align}
Finally, if $g\in H^{\frac{3}{2}}(\Gamma_u)$, then 
 $(u,\pi)\in \mathbb{H}^{2}\times H^{1}(\Dom)/ \mathbb{R}$ and
 \begin{align}
 \label{eq:elliptic_regularity_2}
     \lVert u \rVert_{\Hs^{2}(\Dom;\R^2)}+\lVert \pi \rVert_{H^{1}(\Dom)/\mathbb R}\leq C\lVert g\rVert_{H^{3/2}(\Gamma_u)}.
 \end{align}
\end{theorem}

Next, we denote by $\mathcal{D}$ the solution map defined by \Cref{Regularity of the Dirichlet map} which associate to a boundary datum $g$ the velocity $u$ solution of \eqref{Non Homogeneous Stokes}, i.e.\ $\mathcal{D} g:=u$. 
From the above result, we obtain 

\begin{corollary}\label{Corollary regularity Neumann map}
Let $\mathcal{D}$ and $U$ be the Dirichlet map and a separable Hilbert space, respectively. 
Then
\begin{align*}
   \mathcal{D}\in \mathscr{L}(H^{-\alpha}(\Gamma_u;U),\calL_2(U,\Do(A^{-\frac{\alpha}{2}+\frac{1}{4}})))\quad \text{for } \alpha\in \left[-\frac{1}{2},0\right). 
\end{align*}
\end{corollary}
\begin{proof}
To begin, recall that $H^{s}(\Gamma_u;U)=\calL_2(U,H^s(\Gamma_u))$ for all $s\in \R$, see \Cref{sss:notation}. Hence, due to the ideal property of Hilbert-Schmidt operators, it is enough to consider the scalar case $U=\R$.

By complex interpolation, the estimates in \Cref{Regularity of the Dirichlet map} yield 
$$
\mathcal{D} : H^{2\theta-\frac{1}{2}} (\Gamma_u)\to \Hs^{2\theta}(\Dom)\ \  \text{ for all }\theta\in (0,1).
$$
Hence, the claim now follows from the description of the fractional power of $A$ in \Cref{lem:fractional_powers_stokes}.
\end{proof}
\subsection{Deterministic Navier-Stokes equations}
\label{ss:deterministic_NS}
Let us consider the deterministic Navier-Stokes equations with homogeneous boundary conditions 
\begin{equation}
\left\{
\begin{aligned}
\partial_{t}\overline{u}+\overline{u}\cdot\nabla\overline{u}+\nabla\overline{\pi}  & =\Delta \overline{u}+\overline{f}, \quad &\text{ on }&(0,T)\times \Dom,\\
\operatorname{div}\overline{u}  & =0,&\text{ on }&(0,T)\times \Dom,\\
\overline{u}   & = 0, &\text{ on }&(0,T)\times (\Gamma_b\cup \Gamma_u), \\ 
\overline{u}(0)   & =\overline{u}_0, &\text{ on }&\Dom .  
\end{aligned}
\right.
\label{classical NS}   
\end{equation}
Define the trilinear form 
\begin{equation}
\label{def:b_trilinear_form}
b\left(  u,v,w\right)  =\sum_{i,j=1}^{2}\int_{\Dom}u_{i}
\partial_{i}v_{j}  w_{j}  \, \dd x\dd z=\int_{\Dom}\left(
u\cdot\nabla v\right)  \cdot w\,\dd x\dd z
\end{equation}
which is well--defined and continuous on $\mathbb{L}^{p}\times \mathbb{H}^{1,q}\times
\mathbb{L}^{r}$ by H\"{o}lder's inequality, whenever \begin{align*}
    \frac{1}{p}+\frac{1}{q}+\frac{1}{r}=1.
\end{align*} 
Finally, we introduce the operator
\[
B:\mathbb{L}^{p}\times\mathbb{L}^{r}\rightarrow X_{-1/2,A_{q'}}%
\]
defined by the identity%
\[
\left\langle B\left(  u,v\right)  ,\phi\right\rangle_{X_{-1/2,A_{q'}},X_{1/2,A_q}} =-b\left(  u,\phi
,v\right)  =-\int_{\Dom}\left(  u\cdot\nabla\phi\right)  \cdot v\,\dd x\dd z
\]
for all $\phi\in X_{1/2,A_q}$. Moreover, {if $u\cdot\nabla v\in L^{r}\left(  \Dom;\mathbb{R}^{2}\right)  $ for some $r\in (1,\infty)$}, it
is explicitly given by
\[
B\left(  u,v\right)  =\p(  u\cdot\nabla v).
\]
We have to define our notion of a weak solution for problem \eqref{classical NS}.
\begin{definition}
\label{Def NS determ}Given $\overline{u}_{0}\in H$ and $\overline{f}\in L^{2}\left(  0,T;V^{*}\right)  $, we say that
\[
\overline{u}\in C\left(  \left[  0,T\right]  ;H\right)  \cap L^{2}\left(  0,T;V\right)
\]
is a weak solution of equation \eqref{classical NS} if 
for all $\phi\in \Do\left(  A\right)  $ and $t\in [0,T]$,
\begin{align*}
&  \left\langle \overline{u}\left(  t\right)  ,\phi\right\rangle -\int_{0}^{t}b\left(
\overline{u}\left(  s\right)  ,\phi,\overline{u}\left(  s\right)  \right)  \,\dd s\\
&  =\left\langle \overline{u}_{0},\phi\right\rangle -\int_{0}^{t}\left\langle \overline{u}\left(
s\right)  ,A\phi\right\rangle \,\dd s+\int_{0}^{t}\left\langle \overline{f}\left(  s\right)
,\phi\right\rangle_{V^*,V} \,\dd s.
\end{align*}
\end{definition}
The well-posedness of \eqref{classical NS} in the sense of \Cref{Def NS determ} is a well-known fact. Indeed the following theorem holds, see for instance \cite{lions1996mathematical,temam1995navier,temam2001navier}.
\begin{theorem}
\label{Thm deterministic well--posed}For every $\overline{u}_{0}\in H$ and $\overline{f}\in
L^{2}\left(  0,T;V^*\right)  $ there exists a unique weak solution of
equation \eqref{classical NS}. It satisfies
\[
\lVert \overline{u}\left(  t\right)  \rVert^{2}+2\int_{0}%
^{t}\lVert \nabla \overline{u}\left(  s\right)  \rVert _{L^{2}}^{2}\,\dd s=\lVert
\overline{u}_{0}\rVert^{2}+2\int_{0}^{t}\left\langle \overline{u}\left(  s\right)
,\overline{f}\left(  s\right)  \right\rangle_{V^*,V} \,\dd s.
\]
If $\left(  \overline{u}_{0}^{n}\right)  _{n\in\mathbb{N}}$ is a sequence in $H$
converging to $\overline{u}_{0}\in H$ and $\left(  \overline{f}^{n}\right)  _{n\in\mathbb{N}}$ is a
sequence in $L^{2}\left(  0,T;V^*\right)  $ converging to $\overline{f}\in
L^{2}\left(  0,T;V^*\right)  $, then the corresponding unique solutions
$\left(  \overline{u}^{n}\right)  _{n\in\mathbb{N}}$ converge to the corresponding
solution $\overline{u}$ in $C\left(  \left[  0,T\right]  ;H\right)  $ and in
$L^{2}\left(  0,T;V\right)  $.
\end{theorem}

We end this section with the following lemma, which generalizes \cite[Lemma 1.14]{flandoli2023stochastic} to the $L^p_tL^q_x$ setting. In particular, \cite[Lemma 1.14]{flandoli2023stochastic} corresponds to the specific case $p=q=4.$
\begin{lemma}\label{lemma tecnico}
If {$q>2$} and $p\geq \frac{2q}{q-2}$, $u\in C([0,T];H)\cap L^2(0,T;V),\ v\in L^{p}(0,T;\mathbb{L}^q)$, then
\begin{align}\label{estimate 1 lemma tecnico}
    B(u,v)\in L^2(0,T;V^*),\\
    \label{estimate 4 lemma tecnico} B(v,u)\in L^2(0,T;V^*). 
\end{align}
In particular for each $t\in [0,T],\ \eps,\ \eps'>0$ and $\phi\in L^2(0,T;V)$ it holds
\begin{align}
\label{estimate 2 lemma tecnico}
\int_0^t \lvert\langle u(s)\cdot&\nabla \phi(s), v(s)\rangle\rvert \,\dd s\notag\\ & \leq \eps\lVert\phi\rVert_{L^2(0,t;V)}^2+\eps' \int_0^t \lVert u(s)\rVert_V^2 \,\dd s+\frac{C}{\eps^{\frac{q}{q-2}}{\eps'}^\frac{2}{q-2}}\int_0^t \lVert u(s)\rVert^{2}\lVert v(s)\rVert_{\mathbb{L}^q}^{\frac{2q}{q-2}}\,\dd s,\\
\label{estimate 3 lemma tecnico}
\int_0^t \lvert\langle v(s)\cdot&\nabla \phi(s), u(s)\rangle\rvert \,\dd s\notag \\ & \leq \eps\lVert\phi\rVert_{L^2(0,t;V)}^2+\eps' \int_0^t \lVert u(s)\rVert_V^2 \,\dd s+\frac{C}{\eps^{\frac{q}{q-2}}{\eps'}^\frac{2}{q-2}}\int_0^t \lVert u(s)\rVert^{2}\lVert v(s)\rVert_{\mathbb{L}^q}^{\frac{2q}{q-2}}\,\dd s,
\end{align}
where $C$ is a constant independent from $\eps,\ \eps'.$
\end{lemma}
\begin{proof}
By H\"older inequality, Sobolev embedding theorem and interpolation, for each $\phi\in V$ we have
\begin{align*}
    \lvert\langle B(u(s),v(s)),\phi\rangle\rvert &=\lvert\langle u(s)\cdot\nabla \phi,v(s)\rangle\rvert\\ & \leq \lVert \phi\rVert_V \lVert v(s)\rVert_{\mathbb{L}^q}\lVert u(s)\rVert_{\mathbb{L}^{2q/(q-2)}}\\ & \lesssim_q \lVert \phi\rVert_V \lVert v(s)\rVert_{\mathbb{L}^q}\lVert u(s)\rVert_{\Do(A^{1/q})}\\ & \leq  \lVert \phi\rVert_V \lVert v(s)\rVert_{\mathbb{L}^q}\lVert u(s)\rVert^{1-\frac{2}{q}}\lVert u(s)\rVert_{V}^{\frac{2}{q}}.
\end{align*}
Therefore for each 
\begin{align*}
\int_0^t \lvert \langle u(s)\cdot &\nabla \phi(s),v(s)\rangle \rvert \,\dd s \leq  \eps\lVert\phi\rVert_{L^2(0,t;V)}^2+\frac{C}{\eps}\int_0^t  \lVert v(s)\rVert^2_{\mathbb{L}^q}\lVert u(s)\rVert^{2(1-\frac{2}{q})}\lVert u(s)\rVert_{V}^{\frac{4}{q}} \,\dd s \\ & \leq \eps\lVert\phi\rVert_{L^2(0,t;V)}^2+\frac{C}{\eps}\left(\int_0^t \lVert u(s)\rVert_{V}^{2} \,\dd s\right)^{2/q} \left(\int_0^t \lVert u(s)\rVert^2\lVert v(s)\rVert_{\mathbb{L}^q}^{\frac{2q}{q-2}}\,\dd s \right)^{\frac{q-2}{q}}\\ & \leq  \eps\lVert\phi\rVert_{L^2(0,t;V)}^2+\eps' \lVert u\rVert_{L^2(0,t,V)}^2+\frac{C}{\eps^{\frac{q}{q-2}}{\eps'}^\frac{2}{q-2}}\int_0^t \lVert u(s)\rVert^{2}\lVert v(s)\rVert_{\mathbb{L}^q}^{\frac{2q}{q-2}}\,\dd s. 
\end{align*}
The relation above implies \eqref{estimate 1 lemma tecnico} and \eqref{estimate 2 lemma tecnico}. The proof of  \eqref{estimate 4 lemma tecnico} and \eqref{estimate 3 lemma tecnico} is analogous and we omit the details.
\end{proof}
\subsection{Stochastic convolutions with fractional noise}\label{sec: stoch conv frac noise}
\begin{definition}
Let $U$ be a separable Hilbert space.~A $U$-cylindrical fractional Brownian motion $(W^\hurst(t))_{t\geq 0}$ with Hurst index $\mathcal{H}\in(0,1)$ is defined by the formal series 
\[ W^\hurst(t)=\sum\limits_{n=1}^\infty b^\hurst_n(t) e_n, \]
where $\{e_n\}$ is an orthonormal basis in $U$ and $(b^\hurst(t))_{n\in\mathbb{N}}$ is a sequence of independent standard one-dimensional fractional Brownian motions, i.e.~$\E[b^\hurst_n(t)]=0$ and 
\[ \E[b^\hurst_n(t)b^\hurst_n(s)]=\frac{1}{2} (t^{2\hurst}+s^{2\hurst}-|t-s|^{2\hurst}),~~ s,t\geq 0.\]

\end{definition}
For $\hurst=1/2$ one obtains a cylindrical Brownian motion. However for $\hurst\neq 1/2$ the fbm exhibits a totally different behaviour, in particular is neither Markov nor a semimartingale.

For our aims in \Cref{regularity stokes}, we need the following results on the regularity of stochastic convolutions established in~\cite[Corollary 3.1]{Maslowski} and~\cite[Proposition 11.6]{Maslowski2}. 
\begin{lemma}{\em(\cite[Corollary 3.1]{Maslowski} and~\cite[Proposition 11.6]{Maslowski2})}\label{reg:stochconv}
    Let $A$ be the generator of an analytic $C_0$-semigroup $(S(t))_{t\geq 0}$ on a separable Hilbert space $U_1$, $\Phi\in \mathscr{L}(U,U_1)$. Assume that 
    \begin{align}\label{reg:m}
    \|S(t)\Phi\|_{\calL_2 (U,U_1)}\leq t^{-\gamma} \text{ for } \gamma<\hurst.
    \end{align}
    Then the stochastic convolution $\int_0^t S(t-s)\Phi ~\dd W^\hurst(s)$ has~$\P$-a.s.~$\gamma_1$-H\"older continuous trajectories in $\Do(A^{\gamma_2})$, for $0\leq \gamma_1+\gamma_2<\hurst-{\gamma}$. If $\Phi\in \calL_2(U,U_1)$, then the assumption~\eqref{reg:m}  is satisfied for $\gamma=0$. 
\end{lemma}

\section{Global well-posedness}\label{sec well posed}
Here we prove \Cref{t:global}\eqref{it:global1}. This section is organized as follows. Firstly, in \Cref{sec regularity mild stokes} we prove that the solution $w_g$ of the 2D Stokes equations with boundary noise \eqref{Linear stochastic} satisfies \eqref{eq:pathwise_regularity_wg_intro}.
Secondly, in \Cref{subsec: nonlinear aux}, we prove the existence of a $q$-solution to \eqref{eq:NS_fractional_noise} by studying the auxiliary Navier-Stokes problem \eqref{eq:modified_intro} for a given forcing term $w=w_g$ satisfying the regularity assumption as in \eqref{eq:pathwise_regularity_wg_intro} for a given $q$. Finally, in \Cref{s:uniqueness_q_solution}, we prove the uniqueness of solutions to \eqref{eq:NS_fractional_noise} therefore concluding the proof of \Cref{t:global}\eqref{it:global1}. Recall that ($q$-)solutions of \eqref{eq:NS_fractional_noise} are defined in \Cref{def z weak sol}.

\subsection{Stokes equations}\label{sec regularity mild stokes}
As discussed in \Cref{sec:main results}, we start by considering the linear problem \eqref{Linear stochastic}.
According to \cite{da1993evolution} and \cite[Chapter 15]{DPZ_ergodicity}, the mild solution $w_g$ of the former problem is formally given by
\begin{align}\label{Mild equation linear}
    w_g(t)=A\int_0^t S(t-s)\mathcal{D}[g]\,\dd W^\hurst(s).
\end{align}
Here $A$ is (minus) the Stokes operator with homogeneous boundary conditions as defined in \Cref{subsec:stokes operator }.

Next, we prove that $w_g$ is well-defined in sufficiently regular function spaces therefore allowing us to treat the nonlinearity in the Navier-Stokes equations.

\begin{proposition}\label{regularity stokes}
Let \Cref{ass:ass_fractional} be satisfied.~Then the process $w_g$ is well-defined, progressively measurable, and for all $T>0$ and $\varepsilon>0$,
\begin{equation}
\label{eq:v_process_regularity_H}
w_g\in L^p(\O;C([0,T];\Do(A^{\hurst-\frac{3}{4}-\frac{s}{2}-\varepsilon}))) \ \text{ for all  }p\in (1,\infty).
\end{equation}
In particular, for all $r\in (2, 2q_{\mathcal{H}})$,
\begin{equation}
\label{eq:v_process_regularity_r}
w_g\in C([0,T];\mathbb{L}^{r})\ \text{ a.s.\ }
\end{equation}
\end{proposition}
 \begin{proof}
Note that, thanks to \Cref{Corollary regularity Neumann map}
$$
\D g \in \mathscr{L}_2(U,\Do(A^{\frac{1}{4}-\frac{s+\varepsilon}{2}})).
$$
Hence, by \Cref{reg:stochconv}, a.s.,
$$
w_g=\underbrace{A^{\frac{3}{4}+\frac{s+\varepsilon}{2}}\underbrace{
\int_{0}^{\cdot} S(\cdot-s)A^{\frac{1}{4}-\frac{s+\varepsilon}{2}}\mathcal{D}g\,\dd W^\hurst(s)
}_{\in\, C([0,T];\Do(A^{\hurst-\frac{\varepsilon}{2}}))}}_{\in\, C([0,T];\Do(A^{\hurst-\frac{3}{4}-\frac{s}{2}-\varepsilon}))}.
$$
The arbitrariness of $\varepsilon>0$ yields \eqref{eq:v_process_regularity_H}.~To prove \eqref{eq:v_process_regularity_r}, note that, by \Cref{lem:fractional_powers_stokes},
$$
\Do(A^{\hurst-\frac{3}{4}-\frac{s}{2}-\varepsilon})\subset H^{2\hurst-\frac{3}{2}-s-2\varepsilon}(\Dom;\R^2).
$$
The above space embeds into $L^{r}(\Dom;\R^2)$ for some $r>2$ provided
$$
2\hurst-\frac{3}{2}-s-2\varepsilon>0.
$$
The above is exactly our assumption due to the arbitrariness of $\varepsilon>0$. In particular, by the arbitrariness of $\eps$ and Sobolev's embedding we can choose whatever $r<2\qstar$. \\
     \end{proof}
\begin{remark}[Necessity of the $L^p$-setting for $v$]\label{rk: necessity of Lp setting}
In the setting of \Cref{regularity stokes}, we have
$
    2\hurst-\frac{3}{2}-s<\frac{1}{2}.
$
Therefore, for all choices of $\hurst $ and $s$ in \Cref{ass:ass_fractional}, it follows that
\begin{align*}
    H^{2\hurst-\frac{3}{2}-s-2\varepsilon}(\Dom;\R^2)\not\hookrightarrow L^4(\Dom;\R^2).
\end{align*}
Thus, \eqref{eq:v_process_regularity_r} holds with $r<4$ and therefore 
    $
    B(w_g,w_g)\not \in L^2(0,T;V^*).
    $
    In particular, in the next subsection, we cannot avoid the use of $L^p$-setting in space, cf. the comments below \Cref{ass:ass_w}.
\end{remark}
\begin{remark}
    Previous results with white noise boundary conditions \cite{AL23_boundary_noise,Binz_BN} exploited stochastical maximal $L^p$ regularity techniques to study the linear part of the problem. Here is worth mentioning that we employed the more standard Hilbert value framework because it produces the sharpest result on the regularity of the stochastic convolution in terms of the Hurst parameter $\mathcal{H}$. Indeed, assuming just for simplicity the case $s=0$ and $g\in  L^p(\Gamma_u;U)$ for some $p\in [2,+\infty)$, then by \Cref{Corollary regularity Neumann map}, \cite[Proposition 4.5]{MaslowskiC} and arguing as above we have
    \begin{align*}
        w_g\in C([0,T];\Do(A_p^{\mathcal{H}-1+\frac{1}{2p}-\eps})).
    \end{align*}
    In particular $w_g\in C([0,T];\mathbb{L}^r) $ for some $r>2$ if $\mathcal{H}>1-\frac{1}{2p}$. Therefore the right-hand side is minimized and we can use the rougher noise for $p=2$.
\end{remark}
{We end this subsection showing a lemma concerning the relation between the mild and the weak formulation of \eqref{Linear stochastic} as defined below.
\begin{definition}\label{weak solution linear}
Let \Cref{ass:ass_fractional} be satisfied.
A stochastic process $w$ 
is a weak solution of \eqref{Linear stochastic} if it is $\mathcal{F}$-progressively measurable with $\mathbf{P}-a.s.$ paths in
\begin{align*}
    w_g \in C(0,T;\mathbb{L}^r)
\end{align*}
for some $r\geq 2$, and $ \mathbf{P}-a.s.$ for all $\phi\in \Do(A)$ and $t \in[0,T]$,
\begin{align}\label{weak formulation linear}
\langle w_g(t),\phi\rangle=-\int_0^t\langle w_g(s), A\phi\rangle \,\dd s- \langle g,\hat{n}\cdot\nabla\phi\rangle_{H^{-s}(\Gamma_u),H^{s}(\Gamma_u)}W^\hurst_t.
\end{align}
\end{definition}

As above, $\hat{n}$ denotes the exterior normal vector field on $\partial\Dom$. 
Since $g$ is time-independent, the last term in \eqref{weak formulation linear} can be rewritten as a stochastic integral as
\begin{align*}
    \langle g,\hat{n}\cdot\nabla\phi\rangle_{H^{-s}(\Gamma_u),H^{s}(\Gamma_u)}W^\hurst_t=\int_0^t \langle g,\hat{n}\cdot\nabla\phi\rangle_{H^{-s}(\Gamma_u),H^{s}(\Gamma_u)}\,\dd W^\hurst_s.
\end{align*}

\begin{lemma}\label{lemma weak solution linear equivalence}
    Let \Cref{ass:ass_fractional} be satisfied.
   There exists a unique weak solution of \eqref{Linear stochastic} in the sense of \Cref{weak solution linear} and it is given by the formula \eqref{Mild equation linear}.
\end{lemma}
\begin{proof}
We split the proof into two steps.

\smallskip

\emph{Step 1: There exists a unique weak solution of \eqref{Linear stochastic}  and it is necessarily given by the mild formula \eqref{Mild equation linear}}. Let $\psi \in C^1([0,T];\Do(A))$. Arguing as in the first step of the proof of \cite[Theorem 1.7]{flandoli2023stochastic}, see also \cite[Lemma 3]{LUONGO2024}, one can readily check that $w_g$ satisfies 
\begin{align}\label{weak formulation linear time depending test}
\langle w_g(t),\psi(t)\rangle=& \int_0^t\langle w_g(s),\partial_s\psi(s)\rangle \,\dd s-\int_{0}^{t}\langle w_g(s),A\psi(s)\rangle \,\dd s\notag \\ & -\int_0^t \langle g,\hat{n}\cdot\nabla\psi(s)\rangle_{H^{-s}(\Gamma_u),H^{s}(\Gamma_u)} \,\dd W^{\hurst}_s
\end{align}
for each $t \in[0,T],\ \mathbf{P}-a.s.$
The stochastic integral in the relation above is well-defined as a real-valued stochastic integral. Indeed, recalling that $(W^\hurst_t)_{t\geq 0}$ is a $U$-cylindrical fractional Brownian motion we observe that $\langle g,\hat{n}\cdot\nabla\psi(s)\rangle_{H^{-s}(\Gamma_u),H^{s}(\Gamma_u)} $ is given by the linear operator on  $U$
\[ h'\mapsto \langle g h',\hat{n}\cdot\nabla\psi(s)\rangle_{H^{-s}(\Gamma_u),H^{s}(\Gamma_u)} =L_{\psi}(gh'),   \]
where $L_{\psi}:=\langle \cdot,  \hat{n} \cdot \nabla \psi(\cdot)\rangle_{H^{-s}(\Gamma_u),H^{s}(\Gamma_u)}. $ By the ideal property of the Hilbert-Schmidt operators we have $\calL_2(U, H^{-s}(\Gamma_u))=H^{-s}(\Gamma_u;U)$  and obtain that 
\[ \|  \langle g,\hat{n}\cdot\nabla\psi(s)\rangle_{H^{-s}(\Gamma_u),H^{s}(\Gamma_u)} \|_{U^*} \lesssim \|g\|_{H^{-s}(\Gamma_u)} \|\nabla \psi(s)\|_{H^{s}(\Gamma_u)}~~\text {a.e. on } \Omega\times(0,T). \]
In conclusion, the stochastic integral in~\eqref{weak formulation linear time depending test} is well-defined as a real-valued one (see~\cite[(2.16)]{Maslowski}), since
\begin{align*}
    &\E \Big | \int_0^t \langle g, \hat{n}\cdot \nabla \psi(s) \rangle_{H^{-s}(\Gamma_u), H^{s}(\Gamma_u)} ~\dd W^\hurst_s \Big|^2 \\
    &\leq \hurst(2\hurst-1) \int_0^t \int_0^t \| \langle g, \hat{n}\cdot \nabla \psi(s)  \rangle\|_{U^*}  
  \|  \langle g, \hat{n}\cdot \nabla \psi(v)  \rangle \|_{U^*}
    |s-v|^{2\hurst-2}~\dd s~\dd v\\
    & \leq \hurst(2\hurst-1) \|g\|^2_{H^{-s}(\Gamma_u;U)} \int_0^t\int_0^t  \| \nabla \psi(s)\|_{H^{s}(\Gamma_u)} \|\nabla \psi(v)\|_{H^{s}(\Gamma_u)}|s-v|^{2\hurst-2}~\dd s~\dd v,  
\end{align*}
which is finite since $\psi \in C^1([0,T];\Do(A))$ and $\hurst>1/2$. 
Now consider $\phi\in \Do(A^2)$ and use $\psi_t(s)=S(t-s)\phi,\ s\in [0,t]$ as test function in \eqref{weak formulation linear time depending test}  obtaining
\begin{align}\label{weak formulation linear time depending test 2}
\langle w_g(t),\phi\rangle= & -\int_0^t \langle g,\hat{n}\cdot\nabla S(t-s)\phi\rangle_{H^{-s}(\Gamma_u),H^{s}(\Gamma_u)} \,\dd W^\hurst_s. 
\end{align}
Recalling the definition of the Dirichlet map $\mathcal{D}$, \eqref{weak formulation linear time depending test 2} can be rewritten as 
\begin{align}\label{weak formulation linear time depending test 3}
\langle w(t),\phi\rangle= & \int_0^t \langle \mathcal{D}[g],A S(t-s)\phi\rangle \,\dd W^\hurst_s.
\end{align}
Then, exploiting the self-adjointness property of $S$ and $A$ we have that weak solutions of \eqref{Linear stochastic} satisfy the mild formulation. Therefore they are unique.

\smallskip

\emph{Step 2: The mild formula \eqref{Mild equation linear} is a weak solution of \eqref{Linear stochastic} in the sense of \Cref{weak solution linear}}. We begin by noticing that $w_g$ has the required regularity due to \Cref{regularity stokes}. Let us test our mild formulation \eqref{Mild equation linear} against functions $\phi\in \Do(A^2)$. It holds, exploiting self-adjointness property of $S$ and $A$  \begin{align*}
\langle w(t),\phi\rangle &= \int_0^t \langle \mathcal{D}[g],AS(t-s)\phi\rangle \,\dd W^\hurst_s\\ & =\int_0^t \langle g,\hat{n}\cdot\nabla S(t-s)\phi\rangle_{H^{-s}(\Gamma_u),H^{s}(\Gamma_u)} \,\dd W^\hurst_s\quad \mathbf{P}-a.s.,
\end{align*}
where in the last step we used the definition of Dirichlet map. To complete the proof of this step it is enough to show that 
\begin{align}\label{relation to be verified}
    \int_0^t \langle g,\hat{n}\cdot\nabla S(t-s)\phi\rangle_{H^{-s}(\Gamma_u),H^{s}(\Gamma_u)} \,\dd W^\hurst_s=-\int_0^t \langle w_g(s),A\phi\rangle \,\dd s\notag&\\ 
    + \langle g,\hat{n}\cdot\nabla\phi\rangle_{H^{-s}(\Gamma_u),H^{s}(\Gamma_u)} \, W^\hurst_t \quad \mathbf{P}-a.s. &
\end{align}
The relation \eqref{relation to be verified} is true. Indeed,
\begin{align}\label{weak formulation pre fubini}
    \int_0^t \langle w_g(s),A\phi\rangle \,\dd s&=\int_0^t \,\dd s \int_0^s \langle \mathcal{D}[g],S(s-\tau)A^2 \phi \rangle \,\dd W^\hurst(\tau)\quad \mathbf{P}-a.s.
\end{align}
The double integrals in \eqref{weak formulation pre fubini} can be exchanged via stochastic Fubini's Theorem, see \cite{Nualart1,AlosNualart}.
Therefore the double integral in the right-hand side of \eqref{weak formulation pre fubini} can be rewritten as 
\begin{align*}
    &\int_0^t \,\dd s \int_0^s \langle \mathcal{D}[g],S(s-\tau)A^2 \phi \rangle \,\dd W^\hurst_\tau\\
    &\qquad =\int_0^t \,\dd W^\hurst(\tau) \int_\tau^t  \langle \mathcal{D}[g],S(s-\tau)A^2 \phi \rangle \,\dd s\\ 
    &\qquad = \langle \mathcal{D}[g],A \phi \rangle  W^\hurst_t -\int_0^t \langle \mathcal{D}[g],AS(t-\tau)\phi \rangle \,\dd W^\hurst_\tau\\ 
    &\qquad  = \langle g, \hat{n}\cdot\nabla\phi \rangle_{H^{-s}(\Gamma_u),H^{s}(\Gamma_u)} \, W^\hurst_t\\ 
    & \qquad  -\int_0^t \langle g,\hat{n}\cdot\nabla S(t-\tau)\phi \rangle_{H^{-s}(\Gamma_u),H^{s}(\Gamma_u)} \,\dd W^\hurst_\tau\quad \mathbf{P}-a.s.
\end{align*}
Inserting this expression in \eqref{weak formulation pre fubini}, \eqref{relation to be verified} holds and the proof is complete.
\end{proof}}

\subsection{Auxiliary Navier-Stokes type equations}\label{subsec: nonlinear aux}

Motivated by the auxiliary problem \eqref{eq:modified_intro} and by the results of the previous subsection, here we study the well-posedness of the following abstract PDE
\begin{equation}
\label{eq:modified_NS}
\left\{
\begin{aligned}
\partial_t v &+ A_{q} v + B(v+w,v+w)=0, \qquad t\in [0,T],\\
v(0)&=\uin, 
\end{aligned}
\right.
\end{equation}
with $A_q$ is the Stokes operator on $\Ls^q$ and $B$ is the bilinear nonlinearity as in \Cref{ss:deterministic_NS}. Finally, $q=r/2$ and $(w,r)$ satisfies the following 

\begin{assumption}
\label{ass:ass_w}
$
w\in C([0,T];\mathbb{L}^r)
$ for some $r\in (2,4)$. 
\end{assumption}

Note that the above assumption is satisfied $\P-a.s.$ with $w=w_g$, as it follows from  \Cref{regularity stokes}. Moreover, the limitation $r<4$ is motivated by \Cref{rk: necessity of Lp setting}.
In particular, the arguments used in \cite{AL23_boundary_noise} do not apply to \eqref{eq:modified_NS}.
Indeed, if \Cref{ass:ass_w} holds, then
$$
B(w,w)\not\in L^2(0,T;H^{-1}(\Dom;\R^2)).
$$ 
Hence, the (potential) energy of solutions for \eqref{eq:modified_NS}, i.e.\
$$
\int_{0}^t \int_{\Dom} |\nabla v|^2\,\dd x \,\dd s \ \ \text{ for }\ t>0,
$$ 
is \emph{ill-defined} even in absence of the terms $B(v,v),B(w,v)$ and $B(v,w)$. In particular, one cannot expect energy (or Leray's) type solutions for \eqref{eq:modified_NS} to be defined and the analysis carried on in \cite{AL23_boundary_noise} does not work in our framework. 
However, $B(v,v)\in L^{\infty}(0,T;H^{-1,q}(\Dom;\R^2))$ for some $q>1$ as $r>2$, and therefore $L^{q}$-theory for \eqref{eq:modified_NS} can be built.

\smallskip 

Next, let us describe the main idea behind our construction of a solution to \eqref{eq:modified_NS}. In what follows, the subcriticality of $L^\infty(0,T;\mathbb{L}^r)$ with $r>1$ for the 2D Navier-Stokes equations (cf. the discussion below \Cref{def z weak sol}) plays a central role. 
Indeed, by subcriticality, given $q_0=\frac{r}{2}$, the solution $v_0$ to
$$
\left\{
\begin{aligned}
\partial_t v_0&+A_{q_0} v_0+ B(w,w)=0,\\
 v_0(0)&=0,
\end{aligned}
\right.
$$
satisfies $v_0\in L^p(0,T;L^{r_0}(\Dom;\R^2))$ for some $r_0>r$ and each $p<+\infty$. Hence, we obtained a small gain of space regularity. 
In particular, $\overline{v}_1=v-v_0$ solves
$$
\left\{
\begin{aligned}
\partial_t \overline{v}_1&+A_{q_1} \overline{v}_1+ B(\overline{v}_1,\overline{v}_1)+B(\overline{v}_1,w+v_0)+B(w+v_0,\overline{v}_1)\\
&\qquad\qquad\qquad + B(v_0,w+v_0)+B(w,v_0)=0,\\
 \overline{v}_1(0)&=\uin,
\end{aligned}
\right.
$$
In the above, we would like to take $q_1>q_0$ due to the increased regularity of the forcing terms. Indeed, as  
$v_0\in L^p(0,T;L^{r_0}(\Dom;\R^2))$ for some $r_0>r$ and each $p<+\infty$ one obtains that the terms $B(w+v_0,v_0)$ and $B(w,v_0)$ belong to $L^p(0,T;H^{-1,q_1}(\Dom;\R^2))$ where $\frac{1}{q_1}=\frac{1}{r_0}+\frac{1}{r}$ satisfies $q_1>q_0$. In particular, the terms appearing in the problem above are more regular in space than $B (w,w)$. This opens the door to a further iteration. 
In particular, by considering the solution $v_1$ to
$$
\left\{
\begin{aligned}
\partial_t v_1&+A_{q_1} v_1 + B(v_0,w+v_0)+B(w,v_0)=0,\\
 v_1(0)&=0,
\end{aligned}
\right.
$$
and studying the problem for $\overline{v}_2=\overline{v}_1-v_1$, one can check that the above procedure leads to a further improvement. The idea is to stop the iteration whenever the forcing terms appearing in the procedure are regular enough to build Leray-type solutions to the corresponding PDE. 
Before going further, let us stress that the above procedure is reminiscent of the so-called `DaPrato-Debussche trick' introduced in \cite{da2002two} and now is widely used in the context of stochastic PDEs. 

\smallskip

Let us now turn to the construction of a solution to \eqref{eq:modified_NS}. The above argument motivates the following splitting. Let $N$ be a positive integer such that 
\begin{align}\label{definition of N}
    r\in \left[\frac{2(N+2)}{N+1},\frac{2(N+1)}{N}\right)
\end{align}
then we search of a solution $v$ to \eqref{eq:modified_NS} given by a sum of $N+1$ terms
\begin{align}
\label{eq:v_splitting_proof}
    v=\sum_{i=0}^{N-1} v_i+\overline{v}
\end{align}
where $v_i$ and $\overline{v}$ solve the following system of PDEs on $[0,T]$:
\begin{align}\label{system PDE}
    \left\{
    \begin{aligned}
    \partial_t v_0&+A_{q_0} v_0+B(w,w)=0,\\
    \partial_t v_i&+A _{q_i} v_i+B\Big(v_{i-1},w+\sum_{j=0}^{i-1} v_j\Big)+B\Big(w+\sum_{j=0}^{i-2} v_j,v_{i-1}\Big)=0,\\
    \partial_t \overline{v}&+A \overline{v}+B(\overline{v},\overline{v})+B\Big(\overline{v},w+\sum_{j=0}^{N-1}v_j\Big)+B\Big(w+\sum_{j=0}^{N-1}v_j, \overline{v}\Big)\\ &+B\Big(v_{N-1},w+\sum_{j=0}^{N-1} v_j\Big)+B\Big(w+\sum_{j=0}^{N-2} v_j,v_{N-1}\Big)=0,\\
    v_i(0)&=0,\\
    \overline{v}(0)&=\uin,
    \end{aligned}
    \right.
\end{align}
where, $\sum_{j=0}^{-1} :=0$, $i\in\{1,\dots,N-1\}$ and
\begin{align}\label{def qi}
    q_i=\frac{2r}{r+2+(i+1)(2-r)}.
\end{align}
Note that $v_i$ for $i\in\{1,\dots,N-1\}$ solves a (linear) Stokes problem, while the problem for $\overline{v}$ is a modified version of the Navier-Stokes equations.

At least formally, it is clear that $v$ solves \eqref{eq:modified_NS}. The latter fact is a straightforward consequence of the following identity involving $B(v+w,v+w)$ (letting $v_{-1}=w,\ v_N:=\overline{v}$ for simplicity)
\begin{align*}
B(v+w,v+w)&=\sum_{i,j=-1}^N B(v_i,v_j)\\ &
=
\sum_{i=-1}^N B(v_i,v_i)+
\sum_{i=-1}^{N}\sum_{j=-1}^{i-1} B(v_{i},v_j)
+
\sum_{j=-1}^{N}\sum_{i=-1}^{j-1} B(v_i,v_{j})
\\
& {= B(\overline{v},\overline{v}) +  \sum_{i=-1}^{N-1} B(v_i,v_i)+\sum_{i=-1}^{N-1}B(\overline{v},v_i) + \sum_{i=-1}^{N-1}\sum_{j=-1}^{i-1} B(v_{i},v_j)}\\ &{+\sum_{i=-1}^{N-1}B(v_i,\overline{v})
+
\sum_{j=-1}^{N-1}\sum_{i=-1}^{j-1} B(v_i,v_{j}) }\\
&=
B(\overline{v},\overline{v})+\sum_{i=-1}^{N-1}B(v_i,\overline{v})+\sum_{i=-1}^{N-1}B(\overline{v},v_i)\\ &
+\sum_{i=-1}^{N-1}\sum_{j=-1}^{i} B(v_{i},v_j)
+
\sum_{j=-1}^{N-1}\sum_{i=-1}^{j-1} B(v_i,v_{j})\\ &= \left[B(\overline{v},\overline{v})+B\left(w+ \sum_{j=0}^{N-1} v_i,\overline{v}\right)+B\left(\overline{v},w+\sum_{j=0}^{N-1} v_j\right)\right.\\ &\left. \quad\quad + B\left(v_{N-1},w+\sum_{j=0}^{N-1}v_j\right)+B\left(w+\sum_{j=0}^{N-2}v_j,v_{N-1}\right)\right]\\ & +\sum_{i=-1}^{N-2}\left[B\left(v_{i},w+\sum_{j=0}^{i}v_j\right)+B\left(w+\sum_{j=0}^{i-1}v_{j},v_i\right)\right].
\end{align*}
Looking at the last line, we can formally identify the first bracketed term as the one appearing in the equation for $\overline{v}$ and the $i$-th summand as the one appearing in the equation for $v_{i+1},\ i+1\in \{0,1,\dots, N-1\}.$ To show rigorously that $v$ given in \eqref{eq:v_splitting_proof} with $(v_0,\dots,v_{N-1},\overline{v})$ solving \eqref{system PDE} is a solution to \eqref{eq:modified_NS} we need to check that $v_0,\dots,v_{N-1}$ and $\overline{v}$ are sufficiently regular. 
The appropriate regularity class for $v$ in \eqref{eq:v_splitting_proof} to obtain a solution is given in the following definition, see also \Cref{rmk: integrability} below.

\begin{definition}
\label{def: solution system PDE}
Given $r\in (2,4)$, $N$ given by \eqref{definition of N}, $p\geq 2^N\frac{r}{r-2}$, we say that 
\begin{align*}
(v_0,\dots,v_{N-1},\overline{v})  
\end{align*} is a $(p,r)$-solution of \eqref{system PDE} if
\begin{equation}
    \begin{aligned}\label{regularity v_i v}
    v_i&\in W^{1,{p/2^i}}(0,T;X_{-1/2,A_{q_i}})\cap L^{{p/2^i}}(0,T;X_{1/2, A_{q_i}}),\\
     \overline{v}&\in C([0,T];H)\cap L^2(0,T;V),
\end{aligned}
\end{equation}
where $q_i$ is as in \eqref{def qi},
and for each
\begin{align*}
    (\phi_0,\dots,\phi_{N-1},\overline{\phi})\quad \text{s.t. } \phi_i\in \Do(A_{q_i^{\prime}}),\ \overline{\phi}\in \Do(A)
\end{align*} 
we have, for all $t\in [0,T]$,
\begin{align}
    \langle v_0(t),\phi_0\rangle&=-\int_0^t \langle v_0(s), A_{q_0^{\prime}}\phi_0\rangle \,\dd s+\int_0^t\langle w(s)\otimes w(s),\nabla \phi_0\rangle \,\dd s \label{weak v0},\\
     \langle v_i(t),\phi_i\rangle&=-\int_0^t \langle v_i(s), A_{q_i^{\prime}}\phi_i\rangle \,\dd s+\int_0^t\langle v_{i-1}(s)\otimes v_{i-1}(s),\nabla \phi_i\rangle \,\dd s\label{weak vi}\\ & +\int_0^t \langle v_{i-1}(s)\otimes \left(w(s)+\sum_{j=0}^{i-2} v_j(s)\right),\nabla \phi_i\rangle \,\dd s \notag\\ &+\int_0^t \langle  \left(w(s)+\sum_{j=0}^{i-2} v_j(s)\right)\otimes v_{i-1}(s),\nabla \phi_i\rangle \,\dd s, \notag\\
     \langle \overline{v}(t),\overline{\phi}\rangle &=\langle \uin,\overline{\phi}\rangle-\int_0^t \langle \overline{v}(s), A\overline{\phi}\rangle \,\dd s+\int_0^t \langle\overline{v}(s)\otimes \overline{v}(s),\nabla\overline{\phi}\rangle \,\dd s \label{weak vbar}\\ & +\int_0^t \langle \overline{v}(s)\otimes \left(w(s)+\sum_{j=0}^{N-1}v_j(s)\right),\nabla\overline{\phi}\rangle \,\dd s\notag\\
     & +\int_0^t \langle \left(w(s)+\sum_{j=0}^{N-1}v_j(s)\right)\otimes \overline{v}(s),\nabla\overline{\phi}\rangle \,\dd s\notag\\ 
     &+\int_0^t \langle v_{N-1}(s)\otimes v_{N-1}(s),\nabla \overline{\phi}\rangle \,\dd s \notag\\ & +\int_0^t \langle v_{N-1}(s)\otimes\left(w(s)+\sum_{j=0}^{N-2}v_j(s)\right), \nabla\overline{\phi}\rangle \,\dd s\notag\\ & 
     +\int_0^t \langle \left(w(s)+\sum_{j=0}^{N-2}v_j(s)\right)\otimes v_{N-1}(s), \nabla\overline{\phi}\rangle \,\dd s.\notag
\end{align}
\end{definition}
\begin{remark}\label{rmk: integrability}
We observe that $q_i>1$ for all $i\in \{0,\dots, N-1\}$ and is increasing in $i$. As an immediate consequence of \Cref{def: solution system PDE}, Sobolev embedding theorem and interpolation we have that
\begin{align*}
 v_i\in L^{p/2^{i}}(0,T;\mathbb{L}^{r_i})\cap C([0,T];H),\quad r_i=\frac{2r}{(i+1)(2-r)+2}.  
\end{align*}
In particular, $r_i>2$ for all $i\in \{0,\dots, N-1\}$ and is increasing in $i$. Therefore one can easily check that all the duality pairings in \Cref{def: solution system PDE} are well defined. 
Moreover,  for all $i\in \{1,\dots,N-1\}$,
$$
v_i,\overline{v} \in L^{\frac{2r}{r-2}}(0,T;\mathbb{L}^r).
$$
Indeed, the above assertion for $v_i$ follows from
$p\geq  2^N \frac{r}{r-2}$. While for $\overline{v}$ we can use the standard interpolation inequality $L^2(0,T;H^1)\cap L^\infty(0,T;L^2)\subseteq L^{2/\theta}(0,T;H^{\theta})$ with $\theta=\frac{r-2}{r}\in (0,1)$ and the Sobolev embedding $H^\theta(\Dom)\embed L^r(\Dom)$.

In particular, if $(v_0,\dots,v_{N-1},\overline{v})$ is a $(p,r)$ solution of \eqref{system PDE}, then, given $v:=\sum_{i=0}^{N-1}v_i+\overline{v}$, $u=v+w$ is a $r/2$-solution of \eqref{eq:NS_fractional_noise} in the sense of \Cref{def z weak sol}.
\end{remark}

The following yields the well-posedness of \eqref{eq:modified_NS} in the sense of \Cref{def: solution system PDE}.

\begin{theorem}\label{thm: auxiliary splitting}
Let \Cref{ass:ass_w} be satisfied.
For each $\uin\in H$, $p\geq 2^N\frac{r}{r-2}$ there exists a unique $(p,r)$-solution $(v_0,\dots,v_{N-1},\overline{v})$ of \eqref{system PDE} in the sense of \Cref{def: solution system PDE}. Moreover $\overline{v}$ satisfies the energy relation
\begin{align}\label{eq: energy relation modified}
    \lVert \ov(t)\rVert^2+2\int_0^T\lVert \nabla \ov(s)\rVert_{L^2}^2\,\dd s&=\lVert \uin\rVert^2+2\int_0^t \langle \ov(s)\cdot\nabla\ov(s),w(s)+\sum_{j=0}^{N-1}v_j(s)\rangle \,\dd s\notag\\ &+2\int_0^t \langle v_{N-1}(s)\cdot\nabla\ov(s),w(s)+\sum_{j=0}^{N-1}v_j(s)\rangle \,\dd s\notag\\ &+2\int_0^t \langle (w(s)+\sum_{j=0}^{N-2}v_j(s))\cdot\nabla\ov(s),v_{N-1}(s)\rangle \,\dd s.  
\end{align}
If $\left(\uin^n\right)_{n\in \N}$ is a sequence in $H$ converging to $\uin\in H$ and $(w^n)_{n\in \N}$ is a sequence in $C([0,T];\mathbb{L}^r)$ converging to $w\in C([0,T];\mathbb{L}^r)$, then the corresponding unique solutions $((v_0^n,\dots, v_{N-1}^n, \overline{v}^n))_{n\in \N}$ converge to the corresponding solution $(v_0,\dots, v_{N-1}, \overline{v})$, each one in the topologies of \Cref{def: solution system PDE}.
\end{theorem}
\begin{proof}
    We exploit strongly the triangle structure of \eqref{system PDE} and split the proof in several steps.

    \smallskip
    
    \emph{Step 1: Linear part of \eqref{system PDE}}. We argue by induction and exploit maximal $L^p$ regularity techniques, see \cite[Chapter 3]{pruss2016moving}. The existence and uniqueness of $v_0$ satisfying the corresponding PDE in the sense of \Cref{def: solution system PDE} and the continuous dependence from data, i.e. $w$ in the topology of $C([0,T];\mathbb{L}^r)$, follows if  \begin{align*}
        B(w,w)\in L^p(0,T; X_{-1/2,A_{q_0}}).
    \end{align*}
    The claim is true, indeed $q_0=\frac{r}{2}$ and by H\"older's inequality we have 
    \begin{align*}
        \int_0^T \lVert B(w(s),w(s))\rVert_{X_{-1/2,A_{r/2}}}^p \,\dd s & \leq \int_0^T \lVert w(s)\rVert_{\mathbb{L}^r}^{2p} \,\dd s\leq T \lVert w\rVert_{C([0,T];\mathbb{L}^r)}^{2p}. 
    \end{align*}
    Now assume we have already proved the existence and uniqueness of $(v_i)_{i\in\{0,\dots, l-1\}},$ $ l\leq N-1$ solving \eqref{system PDE} in the sense of \Cref{def: solution system PDE} and depending continuously from the data, i.e. $w$ in the topology of $C([0,T];\mathbb{L}^r)$. Let us check that there exists a unique $v_l$ solving the corresponding PDE in \eqref{system PDE} in the sense of \Cref{def: solution system PDE} and depending continuously from $w$ in the topology of $C([0,T];\mathbb{L}^r)$. Again, due to maximal $L^p$ regularity techniques, it is enough to show that 
   \begin{align*}
        B(v_{l-1},v_{l-1})+B(v_{l-1},w+\sum_{j=0}^{l-2} v_j)-B(w+\sum_{j=0}^{l-2} v_j,v_{l-1})\in L^{p/2^l}(0,T; X_{-1/2,A_{q_l}}).
    \end{align*}
    The claim is true, indeed due to \Cref{rmk: integrability}
    \begin{align*}
        w,\ v_i\in L^{p/2^{l-1}}(0,T;\mathbb{L}^r)\quad \text{if } i\in\{0,\dots,l-2\}  
    \end{align*}
    and by induction hypothesis
    \begin{align*}
        v_{l-1}\in L^{p/2^{l-1}}(0,T;\mathbb{L}^{r_{l-1}}).
    \end{align*}
    Moreover all $v_i,\ i\in {1,\dots, l-1}$ depends continuously from $w\in C([0,T];\mathbb{L}^r)$  in the corresponding topologies.
    Therefore by H\"older's inequality we have
    \begin{align*}
        \int_0^T \lVert B(v_{l-1}&(s),w(s) +\sum_{j=0}^{l-2}v_j(s)) \rVert_{X_{-1/2,A_{q_l}}}^{p/2^l} \,\dd s\\  & \leq \int_0^T \left\lVert v_{l-1}(s)\otimes \left(w(s)+\sum_{j=0}^{l-2}v_j(s) \right)\right\rVert_{\mathbb{L}^{q_l}}^{p/2^l} \,\dd s \notag \\ & \lesssim_{p,l} \int_0^T \lVert v_{l-1}(s)\rVert_{\mathbb{L}^{r_{l-1}}}^{p/2^{l}}\left(\lVert w(s)\rVert_{\mathbb{L}^{r_{l-1}}}^{p/2^{l}}+\sum_{j=0}^{l-2}\lVert v_j(s)\rVert_{\mathbb{L}^{r_{l-1}}}^{p/2^{l}}\right) \,\dd s\\ & \lesssim \lVert v_{l-1}\rVert^{p/2^{l-1}}_{L^{p/2^{l-1}}(0,T;\mathbb{L}^{r_{l-1}})}+\lVert w\rVert^{p/2^{l-1}}_{L^{p/2^{l-1}}(0,T;\mathbb{L}^{q})}+\sum_{j=0}^{l-2}\lVert v_j\rVert^{p/2^{l-1}}_{L^{p/2^{l-1}}(0,T;\mathbb{L}^{r})}.
    \end{align*}
    
    \smallskip
    
    \emph{Step 2: Introduction to the nonlinear part of \eqref{system PDE}}. First, we observe that due to \emph{Step 1} we have that
    \begin{align}\label{eq: estimate correction linear}
    \overline{f}&=-B(v_{N-1},w+\sum_{j=0}^{N-1} v_j)-B(w+\sum_{j=0}^{N-2} v_j,v_{N-1})\in L^2(0,T;V^*),\\
    \label{eq: regularity auxiliary forcing}\Tilde{v}&=w+\sum_{j=0}^{N-1} v_j\in L^{p/2^{N-1}}(0,T;\mathbb{L}^r).
    \end{align}
    Therefore we are left to study the well-posedness in the weak setting of the following PDE
\begin{align}\label{eq: auxiliary pde}
 \left\{
    \begin{aligned}
        \partial_t \overline{v}&+A \overline{v}+B(\overline v, \overline{v})+B(\overline{v},\Tilde{v})+B(\Tilde{v},\overline{v})=\overline{f},\\
        \overline{v}(0)&=\uin.
    \end{aligned}
    \right.
\end{align}
This can be treated similarly to \cite[Section 3.2]{AL23_boundary_noise} and is the object of the remaining steps.

\smallskip

\emph{Step 3: Uniqueness}. Let $\overline{v}^{\left(  i\right)  }$ be two solutions.
The function $z=\overline{v}^{\left(  1\right)  }-\overline{v}^{\left(  2\right)  }$ satisfies
hence%
\begin{align*}
&  \left\langle z\left(  t\right)  ,\overline{\phi}\right\rangle+\int_{0}^{t}\left\langle z\left(  s\right)  ,A\overline{\phi}\right\rangle
\,\dd s -\int_{0}^{t} \langle z(s)\cdot\nabla \overline{\phi},z(s)\rangle  \,\dd s=\int_{0}^{t}\left\langle \widetilde{f}\left(  s\right)  ,\overline{\phi}\right\rangle
\,\dd s
\end{align*}
where
\[
\widetilde{f}=-B\left(  \overline{v}^{\left(  2\right)  }+\Tilde{v},z\right)  -B\left(
z,\overline{v}^{\left(  2\right)  }+\Tilde{v}\right).
\]
By \Cref{lemma tecnico}, $\widetilde{f}\in L^{2}\left(
0,T;V^*\right)  $. Then, by \Cref{Thm deterministic well--posed}%
,
\[
\lVert z\left(  t\right)  \rVert^{2}+2\int_{0}%
^{t}\lVert \nabla z\left(  s\right)  \rVert _{L^{2}}^{2}\,\dd s=2\int%
_{0}^{t} \langle z(s)\cdot\nabla z(s),\overline{v}^{\left(  2\right)  }(s)+\Tilde{v}(s)\rangle    \,\dd s.
\]
Again by \Cref{lemma tecnico}, we have%
\begin{align*}
 \int%
_{0}^{t} \langle z(s)\cdot &\nabla z(s),\overline{v}^{\left(  2\right)  }(s)+\Tilde{v}(s)\rangle    \,\dd s   \\ &
\leq \left\lvert\int%
_{0}^{t} \langle z(s)\cdot \nabla z(s),\overline{v}^{\left(  2\right)  }(s)\rangle    \,\dd s\right\rvert
+\left\lvert \int%
_{0}^{t} \langle z(s)\cdot\nabla z(s),\Tilde{v}(s)\rangle    \,\dd s \right\rvert \\
&  \leq 2\eps\lVert z\rVert_{L^2(0,t;V)}^{2}+\frac{C}{\eps^{3}}\int_0^t\lVert z(s)\rVert
^{2}\lVert \overline{v}^{\left(  2\right)  }(s)\rVert _{\mathbb{L}^4}^{4} \,\dd s\\
&  +2\eps\lVert z\rVert_{L^2(0,t;V)}^2+\frac{C}{\eps^{\frac{r+2}{r-2}}}\int_0^t \lVert z(s)\rVert^{2}\lVert \Tilde{v}(s)\rVert_{\mathbb{L}^r}^{\frac{2r}{r-2}}\,\dd s\\ 
&=4\eps \int_0^t \lVert \nabla z(s)\rVert _{L^2}^{2}\,\dd s+\frac{C}{\eps^{\frac{r+2}{r-2}}%
}\int_0^t \lVert z(s)\rVert^{2}\left(  \lVert \overline{v}^{\left(  2\right)
}(s)\rVert _{\mathbb{L}^4}^{4}+\lVert \Tilde{v}(s)\rVert _{\mathbb{L}^r}^{\frac{2r}{r-2}}\right)\,\dd s.
\end{align*}
Applying the above with $4\eps=\frac{1}{2}$ and renaming the constant $C$, it follows that
\[
\lVert z\left(  t\right)  \rVert^{2}+\int_{0}%
^{t}\lVert \nabla z\left(  s\right)  \rVert _{L^{2}}^{2}\,\dd s\leq C\int%
_{0}^{t}\lVert z\left(  s\right)  \rVert^{2}\left(   \lVert \overline{v}^{\left(  2\right)
}(s)\rVert _{\mathbb{L}^4}^{4}+\lVert \Tilde{v}(s)\rVert _{\mathbb{L}^r}^{\frac{2r}{r-2}}\right)  \,\dd s.
\]
We conclude $z=0$ by the Gr\"onwall lemma, using \eqref{eq: regularity auxiliary forcing} and
the integrability properties of $\overline{v}^{\left(  2\right)  }$.

\smallskip

\emph{Step 4: Existence}. Define the sequence $\left(  \overline{v}^{n}\right)  $ by
setting $\ov^{0}=0$ and for every $n\geq0$, given $\ov^{n}\in C\left(  \left[
0,T\right]  ;H\right)  \cap L^{2}\left(  0,T;V\right)  $, let $\ov^{n+1}$ be the
solution of equation \eqref{classical NS} with initial condition $\uin$ and
with
\[
f=-B\left(  \ov^{n},\Tilde{v}\right)  -B\left(  \Tilde{v},\ov^{n}\right)+\overline{f}.
\] In particular%
\begin{align*}
&  \left\langle \ov^{n+1}\left(  t\right)  ,\ophi\right\rangle+\int_{0}^{t}\left\langle
\ov^{n+1}\left(  s\right)  ,A\ophi\right\rangle \,\dd s -\int_{0}%
^{t}\left\langle  \ov^{n+1}\left(  s\right)  \cdot\nabla\ophi,\ov^{n+1}\left(  s\right)  \right\rangle
\,\dd s\\
&  =\left\langle \uin,\ophi\right\rangle +\int_0^t \langle f(s),\ophi\rangle\, \dd s
\end{align*}
for every $\ophi\in \Do\left(  A\right)  $. The above is well-defined as 
\[
B\left(  \ov^{n},\Tilde{v}\right)  ,B\left(  \Tilde{v},\ov^{n}\right)  ,\overline{f} \in
L^{2}\left(  0,T;V^*\right)
\]
by \Cref{lemma tecnico} and \eqref{eq: estimate correction linear}.

Then let us investigate the convergence of $\left(  \ov^{n}\right)  $. First,
let us prove a bound. From the previous identity and
\Cref{Thm deterministic well--posed} we get
\begin{align*}
&  \lVert \ov^{n+1}\left(  t\right)  \rVert^{2}+2\int%
_{0}^{t}\lVert \nabla \ov^{n+1}\left(  s\right)  \rVert _{L^{2}}%
^{2}\,\dd s\\
&  =\lVert \uin\rVert^{2}  +2\int_{0}^{t}\left(  b\left(  \ov^{n},\ov^{n+1},\Tilde{v}\right)  +b\left(
\Tilde{v},\ov^{n+1},\ov^{n}\right)  +\langle \overline{f}, \ov^{n+1}\rangle  \right)  \left(
s\right)  \,\dd s.
\end{align*}
It gives us using \Cref{lemma tecnico} and \eqref{eq: estimate correction linear}
\begin{align*}
&  \lVert \ov^{n+1}\left(  t\right)  \rVert^{2}+\int%
_{0}^{t}\lVert \nabla \ov^{n+1}\left(  s\right)  \rVert _{L^{2}}%
^{2}\,\dd s\\
&  =\lVert \uin\rVert^{2}+\eps\int_{0}^{t}\lVert
\ov^{n}\left(  s\right)  \rVert _{V}^{2}\,\dd s\\
&  +C_{\eps}\int_{0}^{t}\lVert \ov^{n}\left(  s\right)  \rVert^{2}\lVert \Tilde{v}(s)\rVert_{\mathbb{L}^r}^{\frac{2r}{r-2}}  \,\dd s+C_{\eps}\int_{0}^{t}\lVert \overline{f}\left(  s\right)
\rVert _{V^*}^{2}\,\dd s.
\end{align*}
Choosing a small constant $\eps$, one can find
$R>\lVert \uin\rVert^{2}$ and $\overline{T}$ small enough, depending only from $\lVert \uin\rVert$ and $\lVert \Tilde{v}\rVert_{L^{\frac{2r}{r-2}}(0,T;\mathbb{L}^r)}$, such that
if
\begin{equation}
\sup_{t\in\left[  0,\overline{T}\right]  }\lVert \ov^{n}\left(  t\right)  \rVert
^{2}\leq R,\qquad\int_{0}^{\overline{T}}\lVert \ov^{n}\left(  s\right)  \rVert
_{V}^{2}\,\dd s\leq R \label{fixed point R}%
\end{equation}
then the same inequalities hold for $\ov^{n+1}$.

Set $z_n=\ov^{n}-\ov^{n-1}$, for $n\geq1$. From the identity above,%
\begin{align*}
&  \left\langle z_{n+1}\left(  t\right)  ,\ophi\right\rangle -\int_{0}%
^{t}\left(  b\left(  \ov^{n+1},\ophi,\ov^{n+1}\right)  -b\left(  \ov^{n},\ophi
,\ov^{n}\right)  \right)  \left(  s\right)  \,\dd s\\
&  =-\int_{0}^{t}\left\langle z_{n+1}\left(  s\right)  ,A\ophi\right\rangle
\,\dd s-\int_{0}^{t}\left\langle \left(  B\left(  \ov^{n},\Tilde{v}\right)  -B\left(
\ov^{n-1},\Tilde{v}\right)  \right)  \left(  s\right)  ,\ophi\right\rangle \,\dd s\\
&  -\int_{0}^{t}\left\langle \left(  B\left(  \Tilde{v},\ov^{n}\right)  -B\left(
\Tilde{v},\ov^{n-1}\right)  \right)  \left(  s\right)  ,\ophi\right\rangle \,\dd s.
\end{align*}
Since%
\begin{align*}
&  b\left(  \ov^{n+1},\ophi,\ov^{n+1}\right)  -b\left(  \ov^{n},\ophi,\ov^{n}\right)
-b\left(  z_{n+1},\ophi,z_{n+1}\right) \\
&  =b\left(  \ov^{n},\ophi,z_{n+1}\right)  +b\left(  z_{n+1},\ophi,\ov^{n}\right)
\end{align*}
we may rewrite it as%
\begin{align*}
&  \left\langle z_{n+1}\left(  t\right)  ,\ophi\right\rangle -\int_{0}%
^{t}b\left(  z_{n+1}\left(  s\right)  ,\ophi,z_{n+1}\left(  s\right)  \right)
\,\dd s\\
&  =-\int_{0}^{t}\left\langle z_{n+1}\left(  s\right)  ,A\ophi\right\rangle
\,\dd s-\int_{0}^{t}\left\langle \left(  B\left(  z_{n},\Tilde{v}\right)  +B\left(
\Tilde{v},z_{n}\right)  \right)  \left(  s\right)  ,\ophi\right\rangle \,\dd s\\
&  +\int_{0}^{t}\left(  b\left(  \ov^{n},\ophi,z_{n+1}\right)  +b\left(
z_{n+1},\ophi,\ov^{n}\right)  \right)  \left(  s\right)  \,\dd s.
\end{align*}
One can check as above the applicability of
\Cref{Thm deterministic well--posed} and get
\begin{align*}
&  \lVert z_{n+1}\left(  t\right)  \rVert ^{2}+2\int%
_{0}^{t}\lVert \nabla z_{n+1}\left(  s\right)  \rVert _{L^{2}}%
^{2}\,\dd s\\
&  =2\int_{0}^{t}\left(  b\left(  z_{n},z_{n+1},\Tilde{v}\right)  +b\left(
\Tilde{v},z_{n+1},z_{n}\right)  \right)  \left(  s\right)  \,\dd s\\
&  +2\int_{0}^{t}b\left(
z_{n+1},z_{n+1},\ov^{n}\right)\left(  s\right)  \,\dd s.
\end{align*}
As above, thanks to \Cref{lemma tecnico} we deduce that
\begin{align*}
\int_0^t \lvert b\left(  z_{n+1}%
,z_{n+1},\ov^{n}\right)(s)  \rvert  \,\dd s\leq\frac{1}{4}\int_0^t \lVert z_{n+1}(s)%
\rVert _{V}^{2}\,\dd s+C\int_0^t\lVert z_{n+1}(s)\rVert^{2}\lVert
\ov^{n}(s)\rVert _{\mathbb{L}^4}^{4} \,\dd s.    
\end{align*}
But%
\begin{align*}
&\int_0^t\lvert b\left(  z_{n},z_{n+1},\Tilde{v}\right)(s)  +b\left(  \Tilde{v},z_{n+1}, z_{n}\right)(s)
\rvert \,\dd s\\&\leq\frac{1}{4}\int_0^t\lVert z_{n+1}(s)\rVert _{V}^{2}\,\dd s+\frac{1
}{8}\int_0^t\lVert z_{n}(s)\rVert _{V}^{2} \,\dd s+C\int_0^t\lVert z_{n}(s)\rVert%
^{2}\lVert \Tilde{v}(s)\rVert _{\mathbb{L}^r}^{\frac{2r}{r-2}}\,\dd s.
\end{align*}
Hence%
\begin{align*}
&  \lVert z_{n+1}\left(  t\right)  \rVert ^{2}+\int%
_{0}^{t}\lVert \nabla z_{n+1}\left(  s\right)  \rVert _{L^{2}}%
^{2}\,\dd s\\
&  \leq C\int_{0}^{t}\lVert z_{n+1}\left(  s\right)  \rVert%
^{2} \lVert \ov^{n}\left(  s\right)  \rVert _{\mathbb{L}^4}%
^{4}  \,\dd s\\
&  +\frac{1}{4}\int_{0}^{t}\lVert z_{n}\left(  s\right)  \rVert _{V}%
^{2}\,\dd s+C\int_{0}^{t}\lVert z_{n}\left(  s\right)  \rVert%
^{2}\lVert \Tilde{v}\left(  s\right)  \rVert _{\mathbb{L}^r}^{\frac{2r}{r-2}}\,\dd s.
\end{align*}
Now we work under the bounds \eqref{fixed point R} and deduce, using the Gr\"onwall
lemma, for $\overline{T}$, depending only from $\lVert \uin\rVert$ and $\lVert \Tilde{v}\rVert_{L^{\frac{2r}{r-2}}(0,T;\mathbb{L}^r)}$, possibly smaller than the previous one,%
\begin{align*}
&  \sup_{t\in\left[  0,\overline{T}\right]  }\lVert z_{n+1}\left(  t\right)
\rVert^{2}+\int_{0}^{\overline{T}}\lVert z_{n+1}\left(  s\right)
\rVert _{V}^{2}\,\dd s  \leq\frac{1}{2}\left(  \sup_{t\in\left[  0,\overline{T}\right]  }\lVert
z_{n}\left(  t\right)  \rVert^{2}+\int_{0}^{\overline{T}}\lVert
z_{n}\left(  s\right)  \rVert _{V}^{2}\,\dd s\right).
\end{align*}
Now we can proceed as in the second step of the proof of \cite[Theorem 3.3]{AL23_boundary_noise}, showing that the sequence $\left(  \ov^{n}\right)  $ is Cauchy in $C\left(
\left[  0,\overline{T}\right]  ;H\right)  \cap L^{2}\left(  0,\overline{T};V\right)  $. Its limit
$\ov$ is a weak solution of \eqref{eq: auxiliary pde} on $[0,\overline{T}]$ and , hence, by the previous step, it is the unique solution. We refer the reader to \cite[Theorem 3.3]{AL23_boundary_noise} for further details. After proving existence and uniqueness in $[0,\overline{T}]$ we can reiterate the existence procedure and in a finite number of steps cover the interval $[0,T]$.

\smallskip

\emph{Step 5: Continuity dependence on the data} 
Let $\ov^n$  (resp. $\ov$) the unique solution of \eqref{eq: auxiliary pde} with data $\uin^n,\ \bar{f}^n,\ \Tilde{v}^n$ (resp. $\uin,\ \bar{f},\ \Tilde{v}$). Since $u_0^n\rightarrow u_0$ in $H$  (resp. $\overline{f}^n\rightarrow \overline{f}$ in $L^2(0,T;V^*)$, $\Tilde{v}^n\rightarrow\Tilde{v}$ in $L^{p/2^{N-1}}(0,T;\mathbb{L}^r)$) the family $(\uin^n)_{n\in\mathbb{N}}$ is bounded in $H$ (resp. the family $(\overline{f}^n)_{n\in\mathbb{N}}$ is bounded in $L^2(0,T;V^*)$, the family $(\tilde{v}^n)_{n\in\mathbb{N}}$ is bounded in $L^{p/2^{N-1}}(0,T;\mathbb{L}^r)$), by \eqref{eq: energy relation modified} one can show easily that the family $(\ov^n)_{n\in\mathbb{N}}$ is bounded in $C([0,T];H)\cap L^2(0,T;V).$ Moreover for each $t\in [0,T]$, $z^n=\ov^n-\ov$ satisfies the energy relation
\begin{align}\label{continuity step 1}
    \frac{1}{2}\lVert z^n(t)\rVert^2+\int_0^t \lVert \nabla z^n(s)\rVert_{L^2}^2 \,\dd s&= \frac{1}{2}\lVert \uin^n-\uin\rVert^2 \notag\\ &+\int_0^t b(z^n(s),z^n(s),\ov(s))\,\dd s \notag\\ &+\int_0^t b(\ov^n(s),z^n(s),\tilde{v}^n(s)-\tilde{v}(s))\,\dd s\notag\\ &+\int_0^t b(z^n(s),z^n(s),\tilde{v}(s))\,\dd s\notag\\
    &+ \int_0^t b(\tilde{v}^n(s)-\tilde{v}(s),z^n(s),\ov(s))\,\dd s \notag\\
    & +\int_0^t\langle \overline{f}^n(s)-\overline{f}(s),z^{n}(s)\rangle \,\dd s.
\end{align}
Thanks to \autoref{lemma tecnico}, we can easily bound the right-hand side of relation \eqref{continuity step 1} by Young's inequality and H\"older's inequality obtaining 
\begin{align}\label{continuity step 2}
    \frac{1}{2}\lVert z^n(t)\rVert^2+\frac{1}{2}\int_0^t \lVert \nabla z^n(s)\rVert_{L^2}^2 \,\dd s&\leq \frac{1}{2}\lVert \uin^n-\uin\rVert^2+\int_0^t \lVert \overline{f}^n(s)-\overline{f}(s)\rVert_{V^*}^2 \,\dd s\notag\\ &+C\int_0^t \lVert z^n(s)\rVert^2 \left(\lVert \ov(s)\rVert_{\mathbb{L}^4}^4+\lVert \Tilde{v}(s)\rVert_{\mathbb{L}^{r}}^{\frac{2r}{r-2}}\right)\,\dd s\notag\\ & +C\lVert \tilde{v}^n-\tilde{v}\rVert_{L^{\frac{2r}{r-2}}(0,T;\mathbb{L}^r)}^2 \lVert \ov^n\rVert_{C([0,T];H)}^{\frac{2(r-2)}{r}}\lVert \ov^n\rVert_{L^2(0,T;V)}^{\frac{4}{r}}\notag\\ &+C\lVert \tilde{v}^n-\tilde{v}\rVert_{L^{\frac{2r}{r-2}}(0,T;\mathbb{L}^r)}^2 \lVert \ov\rVert_{C([0,T];H)}^{\frac{2(r-2)}{r}}\lVert \ov\rVert_{L^2(0,T;V)}^{\frac{4}{r}}.
\end{align}
Applying Gr\"onwall's inequality to relation \eqref{continuity step 2} the claimed continuity follows.
\end{proof}
\begin{remark}\label{remark measurability}
Freezing the variable $\omega\in \Omega$ and solving \eqref{eq:modified_NS} for each $\omega$ does not allow us to obtain information about the measurability properties of $v$. However, the measurability of $v$ with respect to the progressive $\sigma$-algebra follows from the continuity of the solution map with respect to $\uin$ and $w$. Therefore we have the required measurability properties for $v$ with $w$ being the mild solution of \eqref{Linear stochastic}. In particular $ v$ has $\mathbf{P}$-a.s. paths in $C(0,T;H)\cap L^{\frac{2q}{q-1}}(0,T;\mathbb{L}^{2q})$ for each $r\in (1,q_{\mathcal{H}})$, it is progressively measurable with respect to these topologies.
\end{remark}
 Combining \Cref{regularity stokes}, \Cref{lemma weak solution linear equivalence}, \Cref{thm: auxiliary splitting} and \Cref{remark measurability} we get immediately the existence of a $q$-solution of equation \eqref{eq:NS_fractional_noise} in the sense of \Cref{def z weak sol} for each $q\in (1,q_\mathcal{H})$.

\subsection{Proof of \Cref{t:global}\eqref{it:global1}}
\label{s:uniqueness_q_solution}
As discussed above, the results of \Cref{sec regularity mild stokes}, \Cref{subsec: nonlinear aux} provide the existence of a $q$-solution of equation \eqref{eq:NS_fractional_noise} in the sense of \Cref{def z weak sol} for each $q\in (1,q_\mathcal{H})$, moreover such a solution is adapted with paths in $C([0,T];H)$ due to \Cref{rmk: integrability}. {Here we are left to discuss the problem of uniqueness. In order to reach our goal, we start providing a lemma which shows equivalence of $q$-solutions in the sense of \autoref{def z weak sol} and those of the form $u=w_g+v$ as described by \autoref{lem:split} below. Then we conclude the proof by providing uniqueness for solutions of the form $u=w_g+v$.

%

\begin{lemma}
\label{lem:split}
Let $T<\infty$, $\uin\in L^0_{\F_0}(\O;\Ls^2)$, $q\in (1,\qstar)$ and assuming \autoref{ass:ass_fractional}. 
Then $u$ is a $q$-solution to \eqref{eq:NS_fractional_noise} in the sense of \Cref{def z weak sol} with paths in $C([0,T];\Ls^2)$ $\P-a.s.$
if and only if $v:=u-w_g$ is progressively measurable with paths in $ L^{2q'}(0,T;\Ls^{2q})\cap C([0,T];\Ls^2)$ $\P-a.s.$ and it solves, for all divergence-free $\varphi\in C^{\infty}(\Dom;\R^2)$ such that 
$\varphi=0$ on $\Gamma_{b}\cup \Gamma_u
$ and a.e.\ $t\in (0,T)$,
\begin{align}
\label{eq:v_split_equation_NS}
&\int_{\Dom} v(x,t)\varphi(x)\,\dd x 
-
\int_{\Dom} \uin(x)\varphi(x)\,\dd x \\
\nonumber
&\qquad\qquad 
=
\int_0^t \int_{\Dom}\big(v\cdot \Delta \varphi + [(v+w_g)\otimes (v+w_g)]:\nabla \varphi\big)\,\dd x\, \dd s .
\end{align}
\end{lemma}
\begin{proof}
    The proof is a trivial consequence of the notion of $q$-solution in \autoref{def z weak sol}, the regularity of $w_g$ in case of \autoref{ass:ass_fractional}, i.e. \autoref{regularity stokes}, and the equivalence between weak and mild solutions for the linear stochastic problem, see \autoref{lemma weak solution linear equivalence}.
\end{proof}
}

{We call a $v$ progressively measurable with paths in $ L^{2q'}(0,T;\Ls^{2q})\cap C([0,T];\Ls^2)$ $\P-a.s.$ satisfying \eqref{eq:v_split_equation_NS} in the sense of \Cref{lem:split} an auxiliary $q$-solution of \eqref{eq:modified_NS}. The uniqueness of auxiliary $q$-solutions and their independence on $q$ is the content of the following result, which, combined with \Cref{lem:split}, concludes the proof of the first item in \Cref{t:global}.}
\begin{proposition}[Uniqueness]
\label{prop:uniqueness}
Let $v_1$ be an {auxiliary} $q_1$-solution of \eqref{eq:modified_NS} and $v_2$ be an {auxiliary} $q_2$-solution of \eqref{eq:modified_NS}. Then $v_1\equiv v_2$.     
\end{proposition}
The uniqueness result in the Ladyzhenskaya–Prodi–Serrin class of \Cref{prop:uniqueness} might be known to experts. Here, for completeness, we provide a relatively short proof relying on maximal $L^p$-regularity techniques which seem not standard even in the absence of noise. 
\begin{proof}[Proof of \Cref{prop:uniqueness}]
We split the proof into two cases.

\smallskip

\emph{Case q$_1=$ q$_2=$ q.}
Letting $\delta:=v_1-v_2$, for all divergence free vector field $\varphi\in C^{\infty}(\Dom;\R^2)$ such that  $\varphi=0$ on $\Gamma_{b}\cup \Gamma_u $ and a.a.\ $t\in (0,T)$, we have
\begin{align*}
\langle \delta(t),\varphi\rangle&-\int_0^t \langle \delta(s),\Delta \varphi\rangle \,\dd s \\
&=\int_0^t b(\delta(s),\varphi,v_1(s)+w_g(s))\,\dd s+\int_0^t b(v_2(s)+w_g(s),\varphi,\delta(s))\,\dd s .
\end{align*}
As $v_i\in L^{2q'}(0,T;L^{2q}(\Dom;\R^2))$ for $i\in \{1,2\}$, we obtain
$$
B(\delta,v_1+w_g),
B(v_2+w_g,\delta)\in L^{q'}(0,T;X_{-1/2,A_q}) 
\ \ \P-a.s.
$$
Hence, by the density of divergence-free $\varphi\in C^{\infty}(\Dom;\R^2)$ such that  $\varphi=0$ on $\Gamma_{b}\cup \Gamma_u $ in the domain of the Stokes operator $A_{q}$
and from the maximal $L^{q}$-regularity of $A_q$, it follows that 
\begin{align*}
\delta
&\in W^{1,q'}(0,T;X_{-1/2,A_q}) \cap L^{q'}(0,T;X_{1/2,A_q})\\
&\subset C([0,T];B^{1-2/q'}_{q,q'}(\Dom;\R^2))\ \ \P-a.s.
\end{align*}
where in the last step we used the trace embedding \cite[Theorem 3.4.8]{pruss2016moving} applied with $A=A_q$.

By real interpolation (see e.g.\ \cite[Chapter 6]{BeLo}), we obtain
\begin{align*}
( B^{1-2/q'}_{q,q'}(\Dom) ,H^{1,q}(\Dom))_{1/2,1}
\embed     
B^{1-1/q'}_{q,1}(\Dom)
\embed     
L^{2q}(\Dom)
\end{align*}
where in the last step we applied the Sobolev embedding and $1-\frac{1}{q'}-\frac{2}{q}=-\frac{1}{q}$. In particular, 
\begin{equation}
\label{eq:interpolation_inequality_uniqueness}
\|f\|_{L^{2q}(\Dom)}\lesssim 
\|f\|_{B^{1-2/q'}_{q,q'}(\Dom)}^{1/2}
\|f\|_{H^{1,q}(\Dom)}^{1/2}
\end{equation}
for all $f$ for which the right-hand side is finite.

Hence, by maximal $L^{q}$-regularity of $A_q$, again the trace embedding \cite[Theorem 3.4.8]{pruss2016moving} as well as the H\"older inequality, there exists a constant $C_0>0$ independent of $v_1,v_2$ and $\delta$ such that, for all $t\in [0,T]$ and $\P-a.s.$,
\begin{align*}
 &\sup_{r\in [0,t]}\|\delta(r)\|_{B^{1-2/q'}_{q,q'}}^{q'}
 + \int_0^t \|\delta(r)\|_{H^{1,q}}^{q'}\,\dd r \\
&\leq C_0 \int_0^t  \big(\max_{i}\|v_i(r)\|_{L^{2q}}^{q'}
+\|w_g(r)\|_{L^{2q}}^{q'}\big)\|\delta(r)\|_{L^{2q}}^{q'}\,\dd r\\
&\leq C_1 \int_0^t  \big(\max_{i}\|v_i(r)\|_{L^{2q}}^{2q'}
+\|w_g(r)\|_{L^{2q}}^{2q'}\big)\|\delta(r)\|_{B^{1-2/q'}_{q,q'}}^{q'}\,\dd r
+\frac{1}{2}\int_0^t \|\delta(r)\|_{H^{1,q}}^{q'}\,\dd r 
\end{align*}
where in the last step we used the Young inequality and \eqref{eq:interpolation_inequality_uniqueness}.

Now the conclusion follows from the Gr\"onwall lemma and the integrability conditions on $v_1,v_2$ and $w_g\in C([0,\infty);L^{2q})$ for all $q<\qstar$ by \Cref{regularity stokes}.

\smallskip

\emph{Case q$_1\neq$ q$_2$.} In the case of $q_1\neq q_2$ we start by observing that by previous case and the results of \Cref{subsec: nonlinear aux}, for $k\in \{1,2\}$, we have that $v_k=\sum_{i=1}^{N_k-1} v_{k,i}+\overline{v}_k$ where $(v_{k,0},\dots,v_{k,N_{k}-1},\overline{v}_k)$
is the $(p_k,r_k)$-solution of \eqref{system PDE} in the sense of \Cref{def: solution system PDE} with $p_k=2^{N_k}\frac{q_k}{q_k-1}$ and $r_k=2q_k$. The claim is then a particular case of \Cref{l:compatibility} below on the compatibility of the $(p,r)$ solutions of \eqref{eq:modified_NS} in the sense of \Cref{def: solution system PDE}.
\end{proof}

\begin{lemma}[Compatibility]
\label{l:compatibility}
    Let $w\in C([0,T];\Ls^{r})
$ for some $r\in (2,4)$ and $2<\tilde{r}\leq r$. If $(v_0,\dots,v_{N-1},\overline{v})$ is a solution of \eqref{system PDE} in the sense of \Cref{def: solution system PDE} with \begin{align*}
    p\geq 2^N\frac{r}{r-2},\qquad q_i=\frac{2r}{r+2+(i+1)(2-r)}
\end{align*} and $(\wt{v}_0,\dots,\wt{v}_{\tilde{N}-1},\tilde{\overline{v}})$ is a solution of \eqref{eq:modified_NS} in the sense of \Cref{def: solution system PDE} with 
\begin{align*}
 \tilde{p}\geq 2^{\tilde{N}}\frac{\tilde{r}}{\tilde{r}-2},\qquad \tilde{q}_i=\frac{2\tilde{r}}{\tilde{r}+2+(i+1)(2-\tilde{r})}   
\end{align*}
then $v=\tilde{v}$.
\end{lemma}
\begin{proof}
    The case of $r=\tilde{r}$ is obvious since in such a case $N=\tilde{N},\ q_{i}=\tilde{q}_i$ and our construction does not rely on the choice of $p$ so far that $p\geq 2^N\frac{r}{r-2}$. In the general case we have two sequences $(v_0,\dots, v_{N-1},\overline{v})$ and $(\tilde{v}_0,\dots, \tilde{v}_{\tilde{N}-1},\tilde{\overline{v}})$. If $N=\tilde{N}$ the claim is still trivial since our construction does not rely on the choice of $p$ so far that $p\geq 2^N\frac{r}{r-2}$ and of the precise choice of the $q_i$ since $S_{\tilde{q}_i}(t)|_{\mathbb{L}^{q_i}}=S_{q_i}(t)$. If $\tilde{N}>N$ arguing as above we have
    \begin{align*}
        v_i=\tilde{v}_i\quad \forall i\in{0,\dots, N-1}
    \end{align*}
    and we are left to show that $\overline{v}=\sum_{i=N}^{\tilde{N}-1} \tilde{v}_i+\tilde{\overline{v}}=:\hat{v}$. Due to previous steps we can assume that $v$ is $(\tilde{p},r)$ solution since $\tilde{p}\geq2^{\tilde{N}}\frac{\tilde{r}}{\tilde{r}-2}>2^{{N}}\frac{{r}}{{r}-2}$. We observe that due to \Cref{def: solution system PDE} and \Cref{rmk: integrability}, 
\begin{align}\label{regularity_all_finite}
    \overline{v},\hat{v}\in C([0,T];H)\cap L^{\frac{2\tilde{r}}{\tilde{r}-2}}(0,T;\mathbb{L}^{\tilde{r}}),\ f:=w+\sum_{i=0}^{N-1}v_i\in L^{\frac{2\tilde{r}}{\tilde{r}-2}}(0,T;\mathbb{L}^{\tilde{r}}).
\end{align}
Therefore either $\overline{v}$ and $\hat{v}$ satisfy for all divergence free vector field $\varphi\in C^{\infty}(\Dom;\R^2)$ such that  $\varphi=0$ on $\Gamma_{b}\cup \Gamma_u $ equation \eqref{weak vbar}. Therefore, denoting by $\delta(t)=\overline{v}-\hat{v}$ we have that $\delta$ satisfies

\begin{align*}
  \left\langle \delta\left(  t\right)  ,\varphi\right\rangle&-\int_{0}^{t}\left\langle z\left(  s\right)  ,\Delta\varphi\right\rangle
\,\dd s\\ & =\int_0^t b(\delta(s),\varphi,\overline{v}(s)+f(s))\,\dd s+\int_0^t b(\hat{v}(s)+f(s),\varphi,\delta(s))\,\dd s.
\end{align*}
Denoting by $\tilde{q}=\frac{\tilde{r}}{2}\in (1,2)$,
due to relation \eqref{regularity_all_finite}
\begin{align*}
    B(\delta,\overline{v}+f),\ B(\hat{v}+f,\delta)\in L^{\tilde{q}'}(0,T;X_{-1/2,A_{\tilde{q}}}).
\end{align*}
Now the proof proceeds as in the first case of \Cref{prop:uniqueness} and we omit the details.
\end{proof}

\section{Interior regularity}\label{sec:interior reg}
As announced at the end of \Cref{sec:main results}, we prove \Cref{t:global}\eqref{it:global2}. To this end, we first prove the interior regularity of $w_g$ and afterwards the one of $v$ by exploiting the decomposition introduced in \Cref{subsec: nonlinear aux}.

\subsection{Stokes equations}\label{sec interior regularity mild stokes}
Let $(v_0,v_1,\dots,v_{N-1},\overline{v})$ be the $(p,r)$-solution to \eqref{system PDE} as defined in \Cref{def: solution system PDE} given by \Cref{thm: auxiliary splitting}.
Let $N_0$ be the $\mathbf{P}$ null measure set where at least one between 
\begin{align*}
    w_{g}&\notin C([0,T];\mathbb{L}^r),\quad v_i\notin W^{1,p/2^i}(0,T;X_{-1/2,A_{q_i}})\cap L^{p/2^i}(0,T;X_{1/2,A_{q_i}}),\\ \overline{v}&\notin C([0,T];H)\cap L^2(0,T;V),
\end{align*} \eqref{weak v0}, \eqref{weak vi}, \eqref{weak vbar}, \eqref{weak formulation linear} is not satisfied. In the following, we will work pathwise in $\Omega\setminus N_0$ even if not specified.
Thanks to the weak formulation guaranteed by \Cref{lemma weak solution linear equivalence} we can easily obtain the interior regularity of the linear stochastic problem \eqref{Linear stochastic}. Indeed, we are exactly in the same position of \cite[Corollary 4.4]{AL23_boundary_noise} and the following holds. We omit the proof as it follows verbatim the one of \cite[Corollary 4.4]{AL23_boundary_noise}.
\begin{lemma}\label{corollary inteorior regularity linear}
        Let \Cref{ass:ass_fractional} be satisfied.
        Let $w_g$ be the unique weak solution of  \eqref{Linear stochastic} in the sense of \Cref{weak solution linear}.  Then, for all $0<t_1\leq t_2<T,$ $x_0\in \Dom $, $\rho>0$ such that $\mathrm{dist}({B(x_0,\rho)}, \partial \Dom)>0$,
    \begin{align*}
        w_g\in C([t_1,t_2], C^{\infty}({B(x_0,\rho)};\mathbb{R}^2)) \quad \mathbf{P}-a.s.
    \end{align*}
\end{lemma}

\subsection{Auxiliary Navier--Stokes equations and proof of \Cref{t:global}\eqref{it:global2}}\label{sec interior regularity nonlinear auxiliary}
To deal with the interior regularity of \eqref{eq:modified_NS} we perform a Serrin type argument, see \cite{lemarie2018navier,serrin1961interior}. 
In contrast to \cite{AL23_boundary_noise}, as $w_g\notin C([0,T];\mathbb{L}^4)$, we cannot work directly on $v$. However, recalling that the solution $v$ to \eqref{eq:NS_fractional_noise} proven in \Cref{s:uniqueness_q_solution} satisfies
\begin{align}
\label{eq:splitting_interior}
    v=\sum_{i=0}^{N-1}v_i+\overline{v}
\end{align} 
where, again, $(v_0,\dots,v_{N-1},\overline{v})$ is the $(p,r)$-solution to \eqref{system PDE}, cf. \Cref{subsec: nonlinear aux}.
The advantage of having the splitting \eqref{eq:splitting_interior} at our disposal is that $v_i$ satisfies a linear problem where the forcing terms only depend on $v_0,\dots,v_{i-1}$. Thus, by \Cref{corollary inteorior regularity linear} and an induction argument, we can prove that $v_i$ is smooth inside $(0,T)\times \Dom$. While to prove the corresponding statement for $\overline{v}$, we can exploit that $\overline{v}$ is a Leray solution (i.e. it has finite energy) and therefore the Serrin regularization can be adjusted to our situation. 

We begin with analyzing the interior regularity of $v_i$ for $i\in \{0,\dots,N-1\}$. 

\begin{lemma}\label{interior regularity vi}
Let \Cref{ass:ass_fractional}, $r\in (2,4)$ and $p\geq 2^N\frac{r}{r-2}$. Let $(v_0,\dots,v_{N-1},\overline{v})$ be the $(p,r)$-solution of \eqref{system PDE} in the sense of \Cref{def: solution system PDE}. Then for all $i\in \{0,\dots, N-1\},$ $0<t_1\leq t_2<T,$ $x_0\in \Dom $, $\rho>0$ such that $\mathrm{dist}({B(x_0,\rho)}, \partial \Dom)>0$,
    \begin{align*}
        v_i\in C([t_1,t_2], C^{\infty}({B(x_0,\rho)};\mathbb{R}^2)) \quad \mathbf{P}-a.s.
    \end{align*}
\end{lemma}
\begin{proof}
As in the first step of the proof of \Cref{thm: auxiliary splitting} we argue by induction exploiting strongly the linear and triangle structure of \eqref{system PDE}. Before starting we observe that, by \cite[Theorem 3.4.8]{pruss2016moving}, it follows that
\begin{align}\label{regularity auxiliary weak}
    v_i\in C([0,T];\mathbb{L}^{q_i})\cap L^{p/{2^i}}(0,T;X_{1/2, A_{q_i}}).
\end{align}

\emph{Step 1: Interior regularity of $v_0$}.
First let us observe that, since $\mathrm{dist}({B(x_0,\rho)}, \partial \Dom)>0,\ 0<t_1\leq t_2<T$, we can find $\eps$ small enough such that $0<t_1-2\eps<t_1\leq t_2<t_2+2\eps<T,\ \mathrm{dist}({B(x_0,\rho+2\eps)}, \partial \Dom)>0$.
As described in \Cref{lemma weak solution linear equivalence}, arguing as in the proof of \cite[Theorem 7]{flandoli2023stochastic}, we can extend the weak formulation satisfied by $v_0$ to time dependent test functions $\phi\in C^1([0,T]; \mathbb{L}^{q_0'})\cap C([0,T];\Do(A_{q_0'}))$ obtaining that for each $t\in [0,T]$
\begin{align*}
    \langle v_0(t),\phi(t)\rangle& =  \int_0^t \langle v_0(s),\partial_s \phi(s)\rangle \,\dd s- \int_{0}^{t}\left\langle v_0\left(
s\right)  ,A_{q_0'}\phi(s)\right\rangle \,\dd s\\ & +\int_{0}^{t}b\left(
w_g\left(  s\right)  ,\phi(s),w_g\left(
s\right)  \right)  \,\dd s \quad \mathbf{P}-a.s.
\end{align*}
Choosing $\phi=-\nabla^{\perp}\chi,\ \chi\in C^{\infty}_c((0,T)\times \Dom)$ in the weak formulation above and denoting by
\begin{align*}
    \omega_0=\operatorname{curl}v_0&\in C([0,T];H^{-1,q_0}(\Dom))\cap L^p(0,T;L^{q_0}(\Dom)), \\ 
    \omega_w=\operatorname{curl}w_g&\in  C([t_1-2\eps,t_2+2\eps], C^{\infty}({B(x_0,\rho+2\eps)}))\quad \mathbf{P}-a.s.
\end{align*} it follows that
\begin{align*}
    -\int_0^t \langle \omega_0(s),\partial_s\chi(s)\rangle+\langle \omega(s),\Delta \chi(s)\rangle \,\dd s &=\int_0^t \langle \operatorname{curl}(w_g(s)\otimes w_g(s)),\nabla \chi(s)\rangle \,\dd s.
\end{align*} 
This means that $\omega_0$ is a distributional solution in $(0,T)\times \Dom$ of the partial differential equation
\begin{align*}
    \partial_t \omega_0
    &=\Delta \omega_0-\operatorname{div}\operatorname{curl}(w_g(s)\otimes w_g(s)).
\end{align*}
Let us consider $\psi_0\in C^{\infty}_c((0,T)\times \Dom)$ supported in $[t_1-\eps,t_2+\eps]\times B(x_0,\rho+\eps)$ such that it is equal to one in $[t_1-\eps/2,t_2+\eps/2]\times B(x_0,\rho+\eps/2)$. Let us denote by ${\omega}^*_0=\omega_0 \psi_0\in L^p(0,T; L^{q_0}(\mathbb{R}^2))$ supported in $[t_1-\eps,t_2+\eps]\times B(x_0,\rho+\eps)$, then ${\omega}^*_0$ is a distributional solution in $(0,T)\times \mathbb{R}^2$ of
\begin{align}\label{distributional solution tilde omega 1 auxiliary weak}
    \partial_t {\omega}^*_0&= \Delta {\omega}^*_0+h_0
\end{align}
with \begin{align*}
 h_0 & =\partial_t\psi_0 \omega_0-2\nabla\psi_0\cdot\nabla \omega_0-\Delta \psi_0 \omega_0-\psi_0 w_g\cdot\nabla\omega_w.
\end{align*}
Due to \Cref{corollary inteorior regularity linear}
\begin{align*}
    h_0\in L^p(0,T;H^{-1,q_0}(\mathbb{R}^2))\quad \mathbf{P}-a.s.
\end{align*} Then, again by maximal $L^p$-regularity techniques for the heat equation (see e.g.\ \cite[Theorem 4.4.4]{pruss2016moving} \cite[Theorems 10.2.25 and 10.3.4]{Analysis2}) and the trace embedding of \cite[Theorem 3.4.8]{pruss2016moving}, \begin{align*}
    {\omega}^*_0\in C([0,T];L^{q_0}(\mathbb{R}^2))\cap L^p(0,T;H^{1,q_0}(\mathbb{R}^2)).
\end{align*}
Therefore,\begin{align*}
  \omega_0\in & C([t_1-\eps/4,t_2+\eps/4],L^{q_0}(B(x_0,\rho+\eps/4)))\\ &\cap L^p(t_1-\eps/4,t_2+\eps/4,H^{1,q_0}(B(x_0,\rho+\eps/4)))\quad \mathbf{P}-a.s.
\end{align*} Introducing $\phi_0\in C^{\infty}_c(B(x_0,\rho+\eps/4))$ equal to one in $B(x_0,\rho+\eps/8)$, since $\omega_0=\operatorname{curl }v_0$, then $\phi_0 v_0$ satisfies
\begin{align}\label{elliptic u auxiliary weak}
    \Delta(\phi_0 v_0)=\nabla^{\perp}\omega_0\phi_0+\Delta \phi_0 v_0+2\nabla\phi_0\cdot\nabla v_0 ,\quad (\phi_0 v)|_{\partial  B(x_0,\rho+\eps/4)}=0.  
\end{align}
From the regularity of $\omega_0$, by standard elliptic regularity theory (see for example \cite[Chapter 4]{Tri83}), it follows that $\phi_0 v_0\in C([t_1-\eps/4,t_2+\eps/4];H^{1,q_0}(B(x_0,\rho+\eps/4);\mathbb{R}^2))\cap L^p(t_1-\eps/4,t_2+\eps/4;H^{2,q_0}(B(x_0,\rho+\eps/4);\mathbb{R}^2))\ \mathbf{P}-a.s$. Therefore, since $\phi_0\equiv 1$ on $B(x_0,\rho+\eps/8)$ \begin{align}\label{preliminary interior regularity v auxiliary weak}
    v_0\in  & C([t_1-{\eps}/{16},t_2+{\eps}/{16}];H^{1,q_0}(B(x_0,\rho+{\eps}/{16});\mathbb{R}^2))\notag \\ &\cap L^{p}(t_1-{\eps}/{16},t_2+{\eps}/{16};H^{2,q_0}(B(x_0,\rho+{\eps}/{16});\mathbb{R}^2))\quad \mathbf{P}-a.s.
\end{align}
Reiterating the argument, i.e. considering for each $j\in \N,\ j\geq 0$, first $\psi_j\in C^{\infty}_c((0,T)\times \Dom)$ supported in $[t_1-\eps/2^{4j}, t_2+\eps/2^{4j}]\times B(x_0,\rho+\eps/2^{4j})$ identically equal to one in $[t_1-\eps/2^{4j+1}, t_2+\eps/2^{4j+1}]\times B(x_0,\rho+\eps/2^{4j+1})$ and $\phi_j\in C^{\infty}_c(B(x_0,\rho+\eps/2^{4j+2}))$ identically equal to one in $B(x_0,\rho+\eps/2^{4j+3})$ we get iteratively that $\mathbf{P}-$a.s.
\begin{align*}
\omega_0\in & C([t_1-\eps/2^{4j+2},t_2+\eps/2^{4j+2}],H^{j,q_0}(B(x_0,\rho+\eps/2^{4j+2})))\\ &\cap L^p(t_1-\eps/2^{4j+2},t_2+\eps/2^{4j+2},H^{j+1,q_0}(B(x_0,\rho+\eps/2^{4j+2})))   \\
v_0\in & C([t_1-\eps/2^{4(j+1)},t_2+\eps/2^{4(j+1)}],H^{j+1,q_0}(B(x_0,\rho+\eps/2^{4(j+1)});\mathbb{R}^2))\\ &\cap L^p(t_1-\eps/2^{4(j+1)},t_2+\eps/2^{4(j+1)},H^{j+2,q_0}(B(x_0,\rho+\eps/2^{4(j+1)});\mathbb{R}^2)). 
\end{align*}
and the claimed interior regularity for $v_0$ follows.

\smallskip

\emph{Step 2: Inductive step}.
Assume that we have already shown that the claim holds for $v_j,\ j\in \{0, l-1\},$ and $l\leq N-1$. Now let us prove that it holds also for $v_l$. Since $\mathrm{dist}({B(x_0,\rho)}, \partial \Dom)>0,\ 0<t_1\leq t_2<T$, we can find $\eps$ small enough such that $0<t_1-2\eps<t_1\leq t_2<t_2+2\eps<T,\ \mathrm{dist}({B(x_0,\rho+2\eps)}, \partial \Dom)>0$.
As described in \Cref{lemma weak solution linear equivalence}, arguing as in the proof of \cite[Theorem 7]{flandoli2023stochastic}, we can extend the weak formulation satisfied by $v_l$ to time dependent test functions $\phi\in C^1([0,T]; \mathbb{L}^{q_l'})\cap C([0,T];\Do(A_{q_l'}))$ obtaining that for each $t\in [0,T]$
\begin{align*}
    \langle v_l(t),\phi(t)\rangle& =  \int_0^t \langle v_l(s),\partial_s \phi(s)\rangle \,\dd s- \int_{0}^{t}\left\langle v_0\left(
s\right)  ,A_{q_l'}\phi(s)\right\rangle \,\dd s\\ & +\int_{0}^{t}b(v_{l-1}(s),\phi(s), w(s)+\sum_{j=0}^{l-1} v_j(s))  \,\dd s\\ & +\int_0^t b(w(s)+\sum_{j=0}^{i-2} v_j(s),\phi(s),v_{i-1}(s))\,\dd s \quad \mathbf{P}-a.s.
\end{align*}
Choosing $\phi=-\nabla^{\perp}\chi,\ \chi\in C^{\infty}_c((0,T)\times \Dom)$ in the weak formulation above and, for $i\in\{0,\dots, l-1\}$, denoting by
\begin{align*}
    \omega_l=\operatorname{curl}v_l&\in C([0,T];H^{-1,q_l}(\Dom))\cap L^{p/2^l}(0,T;L^{q_l}(\Dom)), \\
    \omega_i=\operatorname{curl}v_i&\in  C([t_1-2\eps,t_2+2\eps], C^{\infty}({B(x_0,\rho+2\eps)})),\\
    \omega_w=\operatorname{curl}w_g&\in  C([t_1-2\eps,t_2+2\eps], C^{\infty}({B(x_0,\rho+2\eps)}))\quad \mathbf{P}-a.s.
\end{align*} arguing as in \emph{Step 1} it follows that $\omega_l$ is a distributional solution in $(0,T)\times \Dom$ of the partial differential equation
\begin{align*}
    \partial_t \omega_l
    &=\Delta \omega_l-\operatorname{div}\operatorname{curl}(v_{l-1}(s)\otimes v_{l-1}(s))\\ &-\operatorname{div}\operatorname{curl}\left(v_{l-1}(s)\otimes \left(w_g(s)+\sum_{j=0}^{l-2}v_{j}(s)\right)\right)\\ &-\operatorname{div}\operatorname{curl}\left(\left(w_g(s)+\sum_{j=0}^{l-2}v_{j}(s)\right)\otimes v_{l-1}(s)\right).
\end{align*}
Let us consider $\psi_0\in C^{\infty}_c((0,T)\times \Dom)$ supported in $[t_1-\eps,t_2+\eps]\times B(x_0,\rho+\eps)$ such that it is equal to one in $[t_1-\eps/2,t_2+\eps/2]\times B(x_0,\rho+\eps/2)$. Let us denote by ${\omega}^*_l=\omega_l \psi_0\in L^p(0,T; L^{q_l}(\mathbb{R}^2))$ supported in $[t_1-\eps,t_2+\eps]\times B(x_0,\rho+\eps)$, then ${\omega}^*_l$ is a distributional solution in $(0,T)\times \mathbb{R}^2$ of
\begin{align}\label{distributional solution tilde omega 1 auxiliary weak 2}
    \partial_t {\omega}^*_l&= \Delta {\omega}^*_l+h_l
\end{align}
with \begin{align*}
 h_l & =\partial_t\psi_0 \omega_l-2\nabla\psi_0\cdot\nabla \omega_l-\Delta \psi_0 \omega_l-\psi_0 w_{l-1}\cdot\nabla\omega_{l-1}\\ & -\psi_0 w_{l-1}\cdot\nabla\left(\omega_{w}+\sum_{j=0}^{l-2}\omega_j\right)-\psi_0 \left(w_j+\sum_{j=0}^{l-2}v_j\right)\cdot\nabla \omega_{l-1}.
\end{align*}
Due to \Cref{corollary inteorior regularity linear} and the inductive hypothesis
\begin{align*}
    h_l\in L^p(0,T;H^{-1,q_l}(\mathbb{R}^2))\quad \mathbf{P}-a.s.
\end{align*}
Now we can argue as in \emph{Step 1} obtaining the claim. We omit the easy details.
\end{proof}
Now we are in the position to apply similar ideas of \cite[Section 4.2]{AL23_boundary_noise} for the equation satisfied by $\overline{v}$. For the sake of completeness, we provide some details.

\begin{lemma}\label{Lemma uniform boundendness}
    Let \Cref{ass:ass_fractional}, $r\in (2,4)$ and $p\geq 2^N\frac{r}{r-2}$. Let $(v_0,\dots,v_{N-1},\overline{v})$ be the $(p,r)$-solution of \eqref{system PDE} in the sense of \Cref{def: solution system PDE}. Then, for all $0<t_1\leq t_2<T,$ $x_0\in \Dom $, $\rho>0$ such that $\mathrm{dist}({B(x_0,\rho)}, \partial \Dom)>0$,
    \begin{align*}
        \overline{v}\in C([t_1,t_2], H^{3/2}({B(x_0,\rho)};\mathbb{R}^2)) \quad \mathbf{P}-a.s.
    \end{align*}
\end{lemma}
\begin{proof}
First let us observe that, since $\mathrm{dist}({B(x_0,\rho)}, \partial \Dom)>0,\ 0<t_1\leq t_2<T$, we can find $\eps$ small enough such that $0<t_1-2\eps<t_1\leq t_2<t_2+2\eps<T,\ \mathrm{dist}({B(x_0,\rho+2\eps)}, \partial \Dom)>0$. To simplify the notation let us call
\begin{align*}
    \Tilde{v}&=w+\sum_{j=0}^{N-1} v_j,\ \tilde{\omega}=\operatorname{curl}\tilde{v}.
\end{align*}
As described in \Cref{lemma weak solution linear equivalence}, arguing as in the proof of \cite[Theorem 7]{flandoli2023stochastic}, we can extend the weak formulation satisfied by $\overline{v}$ to time-dependent test functions $\phi\in C^1([0,T]; H)\cap C([0,T];\Do(A))$ obtaining that for each $t\in [0,T]$
\begin{align*}
    \langle \ov(t),\phi(t)\rangle-\langle \uin,\phi(0)\rangle& =  \int_0^t \langle \ov(s),\partial_s \phi(s)\rangle \,\dd s- \int_{0}^{t}\left\langle \ov\left(
s\right)  ,A\phi(s)\right\rangle \,\dd s\\ & +\int_{0}^{t}b\left(
\ov\left(  s\right)  +\tilde{v}\left(  s\right)  ,\phi(s),\ov\left(  s\right)\right)  \,\dd s\\ & +\int_{0}^{t}b\left(
\ov\left(  s\right)   ,\phi(s),\tilde{v}\left(  s\right) \right)  \,\dd s\\ & +\int_{0}^{t}b\left(
v_{N-1}\left(  s\right)   ,\phi(s),\tilde{v}(s) \right)  \,\dd s\\ & +\int_{0}^{t}b\left(
\tilde{v}\left(  s\right)-v_{N-1}\left(  s\right)   ,\phi(s),v_{N-1}\left(  s\right) \right)  \,\dd s \quad \mathbf{P}-a.s.
\end{align*}
Choosing $\phi=-\nabla^{\perp}\chi,\ \chi\in C^{\infty}_c((0,T)\times \Dom)$ in the weak formulation above and, for $i\in \{0,\dots, N-1\},$ denoting by 
\begin{align*}
    \omega=\operatorname{curl}v&\in C([0,T];H^{-1})\cap L^2((0,T)\times \Dom), \\ 
    \omega_i=\operatorname{curl}v_i&\in  C([t_1-2\eps,t_2+2\eps], C^{\infty}({B(x_0,\rho+2\eps)})),\\
    \omega_w=\operatorname{curl}w&\in  C([t_1-2\eps,t_2+2\eps], C^{\infty}({B(x_0,\rho+2\eps)}))\quad \mathbf{P}-a.s.
\end{align*} it follows that \begin{align*}
    -\int_0^t \langle \omega(s),\partial_s\chi(s)\rangle+\langle \omega(s),\Delta \chi(s)\rangle \,\dd s &=\int_0^t \langle \operatorname{curl}(v_{N-1}(s)\otimes \tilde{v}(s)),\nabla \chi(s)\rangle \,\dd s \\ & +\int_0^t \langle \operatorname{curl}(\left(\tilde{v}(s)-v_{N-1}(s)\right)\otimes v_{N-1}(s)),\nabla \chi(s)\rangle \,\dd s\\ &  +\int_0^t\langle \operatorname{curl}(\ov(s)\otimes \tilde{v}(s)),\nabla \chi(s)\rangle  \,\dd s\\ &  +\int_0^t\langle \operatorname{curl}( \tilde{v}(s)\otimes\ov(s)),\nabla \chi(s)\rangle  \,\dd s\\ &+\int_0^t \langle \omega(s), \ov(s)\cdot\nabla \chi(s)\rangle \,\dd s.
\end{align*} 
This means that $\omega$ is a distributional solution in $(0,T)\times \Dom$ of the partial differential equation
\begin{align*}
    \partial_t \omega+\ov\cdot\nabla \omega
    &=\Delta \omega-\operatorname{div}\bigg(\operatorname{curl}(v_{N-1}(s)\otimes \tilde{v}(s))\\
    &+\operatorname{curl}(\tilde{v}(s)-v_{N-1}(s)\otimes v_{N-1}(s))\\
    &+\operatorname{curl}(\tilde{v}(s)\otimes \ov(s))+\operatorname{curl}(\ov(s)\otimes \tilde{v}(s))\bigg).
\end{align*}
Let us consider $\psi\in C^{\infty}_c((0,T)\times \Dom)$ supported in $[t_1-\eps,t_2+\eps]\times B(x_0,\rho+\eps)$ such that it is equal to one in $[t_1-\eps/2,t_2+\eps/2]\times B(x_0,\rho+\eps/2)$. Let us denote by ${\omega}^*=\omega \psi\in L^2((0,T)\times \mathbb{R}^2)$ supported in $[t_1-\eps,t_2+\eps]\times B(x_0,\rho+\eps)$, then ${\omega}^*$ is a distributional solution in $(0,T)\times \mathbb{R}^2$ of
\begin{align}\label{distributional solution tilde omega 1}
    \partial_t {\omega}^*&= \Delta {\omega}^*-\ov\cdot \nabla{\omega}^*-\tilde{v}\cdot\nabla{\omega}^*+h
\end{align}
with \begin{align*}
 h & =\partial_t\psi \omega-2\nabla\psi\cdot\nabla \omega-\Delta \psi \omega+\ov\cdot\nabla\psi \omega+\tilde{v}\cdot\nabla \psi \omega-\psi\left(\tilde{v}-v_{N-1}\right)\cdot\nabla\omega_{N-1}\\ &-\psi \ov\cdot\nabla \tilde{\omega}-\psi v_{N-1}\cdot\nabla \tilde{\omega}.
\end{align*}
Due to \Cref{corollary inteorior regularity linear} and \Cref{interior regularity vi}
the terms \begin{align*}
    &\tilde{v}\cdot\nabla \psi \omega-\psi\left(\tilde{v}-v_{N-1}\right)\cdot\nabla\omega_{N-1}-\psi \ov\cdot\nabla \tilde{\omega}\\ &-\psi v_{N-1}\cdot\nabla \tilde{\omega} \in L^2((0,T)\times \mathbb{R}^2)\quad \mathbf{P}-a.s.
\end{align*}
Therefore $h\in L^2(0,T;H^{-1}(\mathbb{R}^2))+L^1(0,T;L^2(\mathbb{R}^2))\ \mathbf{P}-a.s.$ Then, arguing as in the first step of the proof of \cite[Theorem 13.2]{lemarie2018navier}, the fact that ${\omega}^*$ is a distributional solution of \eqref{distributional solution tilde omega 1} implies that ${\omega}^*\in C([0,T];L^2(\mathbb{R}^2))\cap L^2(0,T;H^1(\mathbb{R}^2)).$
Therefore \begin{align*}
  \omega\in & C([t_1-\eps/4,t_2+\eps/4],L^2(B(x_0,\rho+\eps/4)))\\ &\cap L^2(t_1-\eps/4,t_2+\eps/4,H^1(B(x_0,\rho+\eps/4)))\quad \mathbf{P}-a.s.  
\end{align*} Introducing $\phi\in C^{\infty}_c(B(x_0,\rho+\eps/4))$ equal to one in $B(x_0,\rho+\eps/8)$, since $\omega=\operatorname{curl }\ov$, then $\phi \ov$ satisfies
\begin{align}\label{elliptic u}
    \Delta(\phi \ov)=\nabla^{\perp}\omega\phi+\Delta \phi \ov+2\nabla\phi\cdot\nabla \ov ,\quad (\phi \ov)|_{\partial  B(x_0,\rho+\eps/4)}=0.  
\end{align}
From the regularity of $\omega$, by standard elliptic regularity theory (see for example \cite{ambrosio2019lectures}), it follows that $\phi \ov\in C([t_1-\eps/4,t_2+\eps/4];H^1(B(x_0,\rho+\eps/4);\mathbb{R}^2))\cap L^2(t_1-\eps/4,t_2+\eps/4;H^2(B(x_0,\rho+\eps/4);\mathbb{R}^2))\ \mathbf{P}-a.s$. Therefore, since $\phi\equiv 1$ on $B(x_0,\rho+\eps/8)$ \begin{align}\label{preliminary interior regularity v}
    \ov\in  & C([t_1-{\eps}/{16},t_2+{\eps}/{16}];H^1(B(x_0,\rho+{\eps}/{16});\mathbb{R}^2))\notag \\ &\cap L^2(t_1-{\eps}/{16},t_2+{\eps}/{16};H^2(B(x_0,\rho+{\eps}/{16});\mathbb{R}^2))\quad \mathbf{P}-a.s.
\end{align}
Let us now consider $\hat{\psi}\in C^{\infty}_c((t_1-{\eps}/{16},t_2+{\eps}/{16})\times B(x_0,\rho+{\eps}/{16}))$ such that it is equal to one in $[t_1-{\eps}/{32},t_2+{\eps}/{32}]\times B(x_0,\rho+{\eps}/{32})$. Let us denote by $\hat{\omega}=\omega \hat{\psi}\in C([0,T];L^2(\mathbb{R}^2))\cap L^2(0,T; H^1(\mathbb{R}^2))$ supported in $(t_1-{\eps}/{16},t_2+{\eps}/{16})\times B(x_0,\rho+{\eps}/{16})$, then $\hat{\omega}$ is a distributional solution in $(0,T)\times \mathbb{R}^2$ of
\begin{align}\label{distributional solution tilde omega 2}
    \partial_t \hat{\omega}&= \Delta \hat{\omega}+\hat{h}
\end{align}
with \begin{align*}
 \hat{h} & =-\ov\cdot \nabla\hat{\omega}-\tilde{v}\cdot\nabla\hat{\omega}+\partial_t\hat{\psi} \omega-2\nabla\hat{\psi}\cdot\nabla \omega-\Delta \hat{\psi} \omega+\ov\cdot\nabla\hat{\psi} \omega+\tilde{v}\cdot\nabla \hat{\psi} \omega\\ &-\hat{\psi}\left(\tilde{v}-v_{N-1}\right)\cdot\nabla\omega_{N-1}-\hat{\psi} \ov\cdot\nabla \tilde{\omega}-\hat{\psi} v_{N-1}\cdot\nabla \tilde{\omega}.
\end{align*}
By \Cref{corollary inteorior regularity linear}, \Cref{interior regularity vi} and relation \eqref{preliminary interior regularity v} it follows that \begin{align*}
    \hat{h}\in L^2(0,T;H^{-1/2}(\mathbb{R}^2))\ \mathbf{P}-a.s.
\end{align*}  Therefore $\hat{\omega}\in  C([0,T];H^{1/2}(\mathbb{R}^2))\cap L^2(0,T;H^{3/2}(\mathbb{R}^2)) \ \mathbf{P}-a.s.$ and arguing as above 
\begin{align*}
    \ov\in  & C([t_1-{\eps}/{64},{t}_2+{\eps}/{64}], H^{3/2}({B(x_0,{r}+{\eps}/{64})};\mathbb{R}^2))\\ &\cap L^2({t}_1-{\eps}/{64},{t}_2+{\eps}/{64},  H^{5/2}({B(x_0,{\rho}+{\eps}/{64})};\mathbb{R}^2))\quad \mathbf{P}-a.s.
    \end{align*}
    This concludes the proof of \Cref{Lemma uniform boundendness}.
\end{proof}
\begin{corollary}\label{corollary inteorior regularity nonlinear}
            Let \Cref{ass:ass_fractional}, $r\in (2,4)$ and $p\geq 2^N\frac{r}{r-2}$. Let $(v_0,\dots,v_{N-1},\overline{v})$ be the $(p,r)$-solution of \eqref{system PDE} in the sense of \Cref{def: solution system PDE}.  Then, for all $0<t_1\leq t_2<T,$ $x_0\in \Dom $, $\rho>0$ such that $\mathrm{dist}({B(x_0,\rho)}, \partial \Dom)>0$,
    \begin{align*}
        \ov\in C([t_1,t_2]; C^{\infty}({B(x_0,\rho)};\mathbb{R}^2))\quad \mathbf{P}-a.s.
    \end{align*}
\end{corollary}
\begin{proof}
Since $\mathrm{dist}({B(x_0,\rho)}, \partial \Dom)>0,\ 0<t_1\leq t_2<T$ we can find $\eps$ small enough such that $0<t_1-2\eps<t_1\leq t_2<t_2+2\eps<T,\ \mathrm{dist}({B(x_0,\rho+2\eps)}, \partial \Dom)>0$ and $\psi\in C^{\infty}_c((0,T)\times\Dom)$ supported in $[t_1-\eps,t_2+\eps]\times B(x_0, \rho+\eps)$ such that it is equal to one in $[t_1+\eps/2,t_2+\eps/2]\times B(x_0,\rho+\eps/2)$.
From \Cref{Lemma uniform boundendness} and Sobolev embedding theorem we know that $\ov\in C([t_1-\eps,t_2+\eps];L^{\infty}(B(x_0,{\rho}+\eps);\mathbb{R}^2))\ \mathbf{P}-a.s.$
Denoting, as in \Cref{Lemma uniform boundendness} by
\begin{align*}
    \Tilde{v}&=w+\sum_{j=0}^{N-1} v_j,\ \tilde{\omega}=\operatorname{curl}\tilde{v},\\ 
    \omega=\operatorname{curl}v&\in C([0,T];H^{-1})\cap L^2((0,T)\times \Dom), \\ 
    \omega_i=\operatorname{curl}v_i&\in  C([t_1-2\eps,t_2+2\eps], C^{\infty}({B(x_0,\rho+2\eps)})),\\
    \omega_w=\operatorname{curl}w&\in  C([t_1-2\eps,t_2+2\eps], C^{\infty}({B(x_0,\rho+2\eps)}))\quad \mathbf{P}-a.s.
    \end{align*}
    and $\omega^{*}=\omega \psi\in L^2((0,T)\times \mathbb{R}^2)$ supported in $[t_1-\eps,t_2+\eps]\times B(x_0,\rho+\eps)$, then, arguing as in the proof of \Cref{Lemma uniform boundendness}, it follows that $\omega^{*}$ is a distributional solution in $(0,T)\times B(x_0, \rho+\eps)$ of
\begin{align}\label{distributional solution tilde omega 3}
    \partial_t \omega^{*}&= \Delta \omega^{*}+\tilde{h}
\end{align}
with \begin{align*}
 \tilde{h} & =-\ov\cdot \nabla{\omega}^*-\tilde{v}\cdot\nabla{\omega}^*+\partial_t{\psi} \omega-2\nabla{\psi}\cdot\nabla \omega-\Delta {\psi} \omega+\ov\cdot\nabla{\psi} \omega+\tilde{v}\cdot\nabla {\psi} \omega\\ &-{\psi}\left(\tilde{v}-v_{N-1}\right)\cdot\nabla\omega_{N-1}-{\psi} \ov\cdot\nabla \tilde{\omega}-{\psi} v_{N-1}\cdot\nabla \tilde{\omega}.
\end{align*}
From the regularity of $\omega,\ \ov,\ \tilde\omega,\ \tilde{v},\ \omega_{N-1}, v_{N-1}$, then $\tilde h \in L^2(t_1-\eps,t_2+\eps;H^{-1}(B(x_0,\rho+\eps)))\ \mathbf{P}-a.s$. By standard regularity theory for the heat equation, see for example Step 2 in \cite[Theorem 13.1]{lemarie2018navier}, a solution of \eqref{distributional solution tilde omega 3} with $\tilde h \in L^2(t_1-\eps,t_2+\eps;H^{k-1}(B(x_0,\rho+\eps))),\ k\in\mathbb{N}$, belongs to $C([{t}_1-\eps/2,{t}_2+\eps/2];H^k(B(x_0,{\rho}+\eps/2)))\cap L^2({t}_1-\eps/2,{t}_2+\eps/2;H^{k+1}(B(x_0,{\rho}+\eps/2)))$. Therefore \begin{align*}
     {\omega}^* &\in C([{t}_1-\eps/2,{t}_2+\eps/2];L^2(B(x_0,{\rho}+\eps/2)))\\ &\cap L^2({t}_1-\eps/2,{t}_2+\eps/2;H^{1}(B(x_0,{\rho}+\eps/2)))\quad \mathbf{P}-a.s.
\end{align*} 
which implies 
\begin{align*}
     \omega &\in C([{t}_1+\eps/4,{t}_2-\eps/4;L^2(B(x_0,{\rho}+\eps/4)))\\ &\cap L^2({t}_1-\eps/4,{t}_2+\eps/4;H^{1}(B(x_0,{\rho}+\eps/4)))\quad \mathbf{P}-a.s.
\end{align*} 
since $\psi\equiv 1$ on $({t}_1-\eps/2,{t}_2+\eps/2)\times B(x_0,{\rho}+\eps/2).$ Considering now $\phi\in C^{\infty}_c(\Dom)$ supported on $B(x_0,\rho+\eps/4)$ such that $\phi\equiv 1 $ on $B(x_0,{\rho}+\eps/8)$, since $\operatorname{curl}\ov=\omega$ then 
$\phi \ov$ satisfies 
\begin{align}\label{elliptic v}
    \Delta(\phi \ov)=\nabla^{\perp}\omega \phi+\Delta \phi \ov+2\nabla\phi\cdot\nabla \ov,\quad (\phi \ov)|_{\partial B(x_0,\rho+\eps/4)}=0.  
\end{align}
Since 
\begin{align*}
 \nabla^{\perp}\omega \phi+\Delta \phi \ov+2\nabla\phi\cdot\nabla \ov    &\in C([{t}_1+\eps/4,{t}_2-\eps/4;H^{-1}(B(x_0,{\rho}+\eps/4);\mathbb{R}^2))\\ &\cap L^2({t}_1-\eps/4,{t}_2+\eps/4;L^{2}(B(x_0,{\rho}+\eps/4);\mathbb{R}^2))\quad \mathbf{P}-a.s., 
\end{align*}
by standard elliptic regularity theory (see for example \cite{ambrosio2019lectures}), \begin{align*}
    \phi \ov & \in C([{t}_1+\eps/4,{t}_2-\eps/4;H^1(B(x_0,{\rho}+\eps/4);\mathbb{R}^2))\\ &\cap L^2({t}_1-\eps/4,{t}_2+\eps/4;H^{2}(B(x_0,{\rho}+\eps/4);\mathbb{R}^2))\quad \mathbf{P}-a.s..
\end{align*}
Since $\phi \equiv 1$ on $B(x_0,{\rho}+\eps/8)$ then 
\begin{align*}
\ov& \in C([{t}_1+{\eps}/{16},{t}_2-{\eps}/{16};H^{1}(B(x_0,\rho+{\eps}/{16})))\\ & \cap L^2({t}_1-{\eps}/{16},{t}_2+{\eps}/{16};H^2(B(x_0,{\rho}+{\eps}/{16})))\quad \mathbf{P}-a.s.    
\end{align*}
Reiterating the argument as in Step 3 in \cite[Theorem 13.1]{lemarie2018navier} the claim follows.
\end{proof}

\begin{proof}[Proof of \Cref{t:global}\eqref{it:global2}]

    The claim follows by \Cref{corollary inteorior regularity linear}, \Cref{interior regularity vi} and \Cref{corollary inteorior regularity nonlinear} and a localization argument. To begin, recall  from the proof of \Cref{t:global}\eqref{it:global1} in \Cref{s:uniqueness_q_solution} that there exists a solution \eqref{eq:NS_fractional_noise} on the time interval $[0,T+1]$ and it is given by $\wt{u}=w_g +\sum_{i=0}^{N-1} v_i+\overline{v}$ where $(v_0,\dots,v_{N-1},\overline{v})$ is the $(p,r)$-solution to \eqref{system PDE} on $[0,T+1]$ for $r<2\qstar$, $N$ as in \eqref{definition of N} and $p\geq 2^N \frac{r}{r-2}$.
    Then, by  \Cref{corollary inteorior regularity linear}, \Cref{interior regularity vi}, \Cref{corollary inteorior regularity nonlinear} and a standard covering argument, for all $t_0\in (0,T)$, $\Dom_0\subset \Dom $ such that $\mathrm{dist}({\Dom_0}, \partial \Dom)>0$,
    \begin{align}
    \label{eq:regularity_u_tilde}
        \widetilde{u}\in C([t_0,T];C^{\infty}(\Dom_0;\mathbb{R}^2))\quad \mathbf{P}-a.s.
    \end{align} 
Now, let $u$ be the unique solution \eqref{eq:NS_fractional_noise} provided by  \Cref{t:global}\eqref{it:global1} on $[0,T]$. By uniqueness, we have $u=\widetilde{u}|_{[0,T]}$ and the conclusion follows from \eqref{eq:regularity_u_tilde}.
\end{proof}

\begin{acknowledgements}
The authors thank Robert Denk and Tim Seitz for several useful discussions. Finally, the authors thank the anonymous referees for helpful comments which improved the paper from its initial version.
\end{acknowledgements}

\end{document}